\documentclass[12pt,a4paper,eqrno]{article}
\usepackage[latin1,utf8]{inputenc}

\usepackage[english]{babel}
\usepackage[T1]{fontenc}
\usepackage{subfig}
\usepackage[margin=1cm,labelfont=bf]{caption}
\usepackage{graphicx}

\usepackage{mathtools} 
\mathtoolsset{showonlyrefs,showmanualtags}
\numberwithin{equation}{section}

\usepackage{amsmath,amssymb,amsthm,mathtools,dsfont,mathrsfs}
\usepackage{graphicx}
\usepackage{enumitem} 
\usepackage{tikz}
\usepackage{tikz-cd}
\usepackage{lmodern}
\usepackage{pgfplots}
\pgfplotsset{compat=1.17}
\usetikzlibrary{intersections}
\usepackage{relsize}
\usepackage{dsfont}
\usepackage{xcolor}

\usepackage{authblk}

\usepackage[top=3cm,bottom=3cm,left=2cm,right=2cm,marginparwidth=0cm]{geometry}
\usepackage[colorlinks=true]{hyperref}
 \hypersetup{
     colorlinks,
     citecolor={green!40!black},
     linkcolor={red!50!black},
     urlcolor={blue!50!black}
 }
\definecolor{mygreen}{RGB}{37,170,37}
\definecolor{myteal}{RGB}{6,144,179}
\definecolor{myblue}{RGB}{59,100,189}
\definecolor{myorange}{RGB}{240,153,85}
\definecolor{myred}{RGB}{190,15,15}
\definecolor{myyellow}{RGB}{255,220,130}
\definecolor{mypurple}{RGB}{125,0,125}

\newcommand{\R}{\mathbb{R}}
\newcommand{\C}{\mathbb{C}}

\newcommand{\PP}{\mathbb{P}}

\DeclareMathOperator{\conv}{conv}

\newcommand{\lprod}{\scalebox{1.2}{\guilsinglleft}}
\newcommand{\rprod}{\scalebox{1.2}{\guilsinglright}}

\newtheorem{theorem}{Theorem}[section]
\newtheorem{proposition}[theorem]{Proposition}
\newtheorem{corollary}[theorem]{Corollary}

\theoremstyle{definition}
\newtheorem{remark}[theorem]{Remark}
\newtheorem{example}[theorem]{Example}

\DeclareMathOperator{\Pic}{Pic}
\DeclareMathOperator{\trace}{trace}
\DeclareMathOperator{\rank}{rank}

\title{Convex hulls of surfaces in fourspace}
\author[1]{Chiara Meroni}
\author[2]{Kristian Ranestad}
\author[3]{Rainer Sinn}
\affil[1]{MPI MiS Leipzig, Leipzig, Germany; \url{chiara.meroni@mis.mpg.de}}
\affil[2]{Universitetet i Oslo, Oslo, Norway; \url{ranestad@math.uio.no}}
\affil[3]{Universit\"at Leipzig, Leipzig, Germany; \url{rainer.sinn@uni-leipzig.de}}
\date{}

\begin{document}
\maketitle

\begin{abstract}
    This is a case study of the algebraic boundary of convex hulls of varieties. We focus on surfaces in fourspace to showcase new geometric phenomena that neither curves nor hypersurfaces do. Our method is a detailed analysis of a general purpose formula by Ranestad and Sturmfels in the case of smooth real algebraic surfaces of low degree (that are rational over the complex numbers). We study both the complex and the real features of the algebraic boundary of Veronese, Del Pezzo and Bordiga surfaces.
\end{abstract}


\section{Introduction}\label{sec:intro_ch}
Convex algebraic geometry is the study of convex semi-algebraic sets. The Zariski closure of the Euclidean boundary of such a semi-algebraic convex body is an algebraic hypersurface, called its algebraic boundary. This hypersurface is a central object, generalizing the hyperplane arrangement associated to a polytope, see~\cite{sinn:algbound}. Information regarding the algebraic boundary of a convex body can help in the study of the convex body itself. 
As for polytopes, it is a hard problem to describe the algebraic boundary given the extreme points of the convex body. This is the problem that we study in this paper. There is a general formula for the algebraic boundary of a convex hull of a variety in terms of projective duality due to Ranestad and Sturmfels \cite{RanStu:ConvexHullVariety} inspired by their previous work on convex hulls of space curves \cite{RanStu:ConvexHullCurves}.
The formula gives an inclusion of the algebraic boundary of the convex hull of a variety, when it is a full dimensional convex compact set, into the union of some varieties $X^{[k]}$ that we will define below.
Here, we take this work further into higher dimension: we consider convex hulls of surfaces in four-dimensional space and determine irreducible components that actually contribute to their algebraic boundaries, and their degrees.

Our approach has two very different steps:
The first step concerns the complex picture. For a projective smooth surface $X$, the  loci $X^{[k]}$ are in general not well understood.  
For small $k$ their degrees are computed by  Kazarian's formulae.  For our surfaces a more detailed study of the loci $X^{[k]}$ allows us to compute the degree of the dual hypersurfaces $\left( X^{[k]} \right)^*$.
In the second step, we look at the real geometry of the surfaces. This involves the topology of the real locus of the surface as well as their different real forms (if known). For instance, we will consider curvature of hypersurfaces as a tool to determine the actual algebraic boundary of the convex body. To understand the real picture, our arguments are more ad hoc as they need to be adapted to analyze different real forms for the same complex algebraic surface. We do this in detail in particular for Del Pezzo surfaces. They are complete intersections of two quadrics and the signatures of the quadrics in this pencil are the main feature that determines the algebraic boundary of the convex hull. 

Our goal is to provide a case study of these phenomena for smooth surfaces of low degree that are rational as complex algebraic varieties.
Those are the Veronese surfaces, the Del Pezzo surfaces (both of degree 4 in $\PP^4$), and the Bordiga surfaces (of degree 6). Our restriction to even degrees comes from real considerations: we are interested only in those cases in which the convex hull of the surface is a full-dimensional convex body, hence compact.
For a compact real Veronese surface, the algebraic boundary of the convex hull is irreducible and of degree $6$, see Theorem~\ref{thm:veronese}. The algebraic boundary of the convex hull of a compact real Del Pezzo surface is the union of two quadratic cones over a nonsingular quadric in $\PP^3$, see Theorem~\ref{thm:del_pezzo}.
For generic Bordiga surfaces, the picture becomes rather involved: we expect that the algebraic boundary can have at most two irreducible components. We were not able to determine the degree exactly. However, one irreducible component has degree 384. Our results in this case are summarized in Theorem~\ref{thm:bordiga}. In this case, we focus only on the real forms arising as blow ups of the real projective plane in a real set of $10$ points.
The computation of the degrees rely on various results in (enumerative) complex algebraic geometry. Most important are the formulae by Kazarian from \cite{Kazarian:Multisingularities} and Vainsencher from \cite{Vainsencher:EnumerationTangentHyperplanes} counting singular plane algebraic curves with fixed degree and prescribed types of singularities.

\subsection{Setting}
Here, we begin with the setting in complex algebraic geometry. The discussion of the real geometry will be done for each surface case separately. Let $X \subset\C^d$ be a complex algebraic variety with a compact real part. We consider the convex body $\conv X \subset \R^d$, the convex hull of the real part of $X$. This is a semi-algebraic set by quantifier elimination. The (projective) Zariski closure of its topological boundary $\partial( \conv X)$ is thus a hypersurface, called the \emph{algebraic boundary} of $\conv X$, see~\cite{sinn:algbound}. We denote it by $\partial_a (\conv X)\subset\PP^d$. When $X$ is smooth, \cite[Theorem 1.1]{RanStu:ConvexHullVariety} provides all the irreducible components of $\partial_a (\conv X)$, that we now recall.

Let $X\subset \PP^d$ be a projective variety, not contained in any hyperplane. Fix an integer $k\leq d$. Define $X^{[k]}$ as the Zariski closure of the set
\begin{equation}
\{ u \in (\PP^d)^* \,|\, u^\perp \hbox{ is tangent to } X \hbox{ at } k \hbox{ linearly independent points  }\} \subset (\PP^d)^*
\end{equation}
of hyperplanes tangent to our variety at $k$ linearly independent points. These $k$ points are singularities in the intersection $u^\perp \cap X$, where $u^\perp = \{[x_0,\ldots,x_d]\in\PP^d \,|\, u_0 x_0 + \ldots + u_d x_d = 0\}$. 

We write $\conv X$ for the convex hull of the real points of $X$ in an affine chart $\R^4\subset \PP^4(\R)$. We assume that this real algebraic set is compact. Then by \cite[Theorem 1.1]{RanStu:ConvexHullVariety}
\begin{equation}\label{eq:formula_boundary_ch}
\partial_a (\conv X) \subset \bigcup_{k = 1}^d \left( X^{[k]} \right)^*.
\end{equation}
There are various ways in which one can improve the inclusion in \eqref{eq:formula_boundary_ch}. For instance, since $\partial_a (\conv X)$ is a hypersurface, we can get rid of the irreducible components of $\left( X^{[k]} \right)^*$ that have higher codimension. A necessary condition (based on Terracini's Lemma, see \cite[Section 1]{RanStu:ConvexHullVariety}) for $\left( X^{[k]} \right)^*$ to be a hypersurface is that the $k$-th secant variety of $X$ has at most codimension $1$, i.e., that $k\geq \left\lceil\frac{n}{\dim X + 1}\right\rceil$. Another issue could be that some of the components in the right hand side of \eqref{eq:formula_boundary_ch} do not have real points. In this case they do not contribute to the algebraic boundary either.
Our goal is to understand which are the relevant irreducible components of the algebraic boundary in the particular case of smooth surfaces in fourspace. We approach this question by analyzing smooth (or even generic) surfaces of fixed low degree and sectional genus. Since we want to compute convex hulls, we need surfaces with a compact real locus, and this can happen only if $\deg X$ is even.
Degree $2$ surfaces are embeddings of $\PP^1\times\PP^1$ in $\PP^3\subset \PP^4$, therefore they have been already investigated in \cite[Section 2.3]{RanStu:ConvexHullVariety}. 
The first interesting degree is then $4$;
here we find two families of smooth surfaces: the Veronese and the Del Pezzo surfaces. A Veronese surface has sectional genus $0$ and it is a projection into $\PP^4$ of the image of the Veronese map $\nu_2(\PP^2)\subset \PP^5$. A Del Pezzo surface has sectional genus $1$ and it is the complete intersection of two quadrics or, equivalently, the blow up of $\PP^2$ in five generic points. In degree $6$ we find Bordiga and K3 surfaces, with sectional genus $3$ and $4$ respectively. Bordiga surfaces are the blow up of $\PP^2$ in ten generic points, whereas K3 surfaces in fourspace are the complete intersection of a quadric and a cubic hypersurface. 
This is a classification as complex algebraic varieties. These surfaces can have other real forms. For Del Pezzo surfaces, we discuss the complete classification of the real forms. In the case of Bordiga surfaces, we only discuss the blow ups of the real projective plane in a real set of $10$ points because of the absence of a complete classification of their real forms.

Using \cite{Vainsencher:EnumerationTangentHyperplanes} or \cite{Kazarian:Multisingularities}, we can compute the degrees of the varieties $X^{[k]}$, as shown in Table \ref{table:degreesXk}. The two papers provide formulae to compute the number of singular curves on a surface, having certain types of singularities. The formulae cannot be applied to the case of the Veronese surface; this requires an ad hoc geometric study of the varieties $X^{[k]}$. 
\begin{table}[ht!]
    \renewcommand{\arraystretch}{1.7}
    \centering
    \begin{tabular}{c||c|c|c|c|}
    \textbf{surface} & \textbf{deg}$X^{[1]}$ & \textbf{deg}$X^{[2]}$ & \textbf{deg}$X^{[3]}$ & \textbf{deg}$X^{[4]}$ \\
    \hline
    \hline
    Veronese & $3$ & $6$ & $6$ & $\emptyset$ \\ 
    \hline
    Del Pezzo & $12$ & $26$ & $40$ & $40$ \\
    \hline
    Bordiga & $27$ & $235$ & $875$ & $1761$ \\
    \hline
    K3 & $42$ & $672$ & $5460$ & $25650$ \\
    \hline
    \end{tabular}
    \caption{The degrees of the $X^{[k]}$'s for low degree surfaces in $\PP^4$.}
    \label{table:degreesXk}
\end{table}
A general fact is that $X^{[1]}$ is the dual variety $X^*$, therefore by the biduality theorem \cite[Theorem 15.24]{Harris:AlgebraicGeometry},
$\left( X^{[1]} \right)^* = X$ is not a hypersurface in our framework. The values of $k$ for which $X^{[k]}$ might be relevant for us, are thus $k=2,3,4$.
K3 surfaces are the only irrational surfaces in the range that we are considering. There is a whole literature regarding the associated varieties $X^{[k]}$ when $X$ is a K3 surface (see for instance \cite{BruLel:SeveriK3}), but not enough for us to compute the degrees of $\left( X^{[k]} \right)^*$. Hence, we focus here on the rational cases of the Veronese, Del Pezzo, and Bordiga surfaces.
The strategy is to partition the numbers in Table \ref{table:degreesXk} with one summand for each irreducible component. We then investigate each irreducible component separately, in order to determine what will be relevant after dualization. When we are dealing with blow ups of the plane from a complex point of view, studying the family of hyperplanes tangent to $X$ at a certain number of points, translates into studying the family of curves in $\PP^2$ passing through the points that we blow up, and having prescribed singularities.

\paragraph{Kazarian and Vainsencher's numbers.}
Let $X$ be a surface and consider a linear system $|D|$ of divisors on $X$. Then \cite[Section 10]{Kazarian:Multisingularities} and \cite[Section 5]{Vainsencher:EnumerationTangentHyperplanes} provide the recipes for computing the number of singular curves in the linear system $|D|$ with prescribed singularities, and going through the correct number of points, in such a way that the answer is finite. For instance, consider the surface $X=\PP^2$ and suppose we want to compute the number of rational quartics in $X$. Denoting by $L$ the class of a line in $\PP^2$, we have $D=4L$. Plane quartics constitute a $\PP^{14}$. In order for the quartics to be rational, they need to have generically three nodes: these are our prescribed singularities. Each node condition drops the dimension of this variety by $1$, thus the variety of rational quartics has codimension $3$ in $\PP^{14}$. We want to compute the degree of this variety, so we impose $11$ linear conditions, namely that these curves pass through $11$ points in $\PP^2$ in general position. We obtain in this way a finite set, whose cardinality is computed by Kazarian and Vainsencher's formulae.
Vainsencher computes singular curves with a certain number of nodes, whereas Kazarian allows also other types of singularities. In particular, $A_1$ denotes a node and $A_2$ denotes a cusp. These will be our cases of interest. We will mainly talk about Kazarian's formulae, to avoid jumping between one paper and the other when we need to count cusps. There are four parameters to be considered
\begin{equation}\label{eq:Kaz_quantities}
\begin{aligned}
    &\mathsf{d} = D.D, && \mathsf{k} = D.K_X, \\
    &\mathsf{s} = K_X.K_X, && \mathsf{x} = \chi(X),
\end{aligned}
\end{equation}
where $K_X$ is the canonical divisor on $X$ and $\chi$ denotes the topological Euler characteristic. 
Here `\,.\,' denotes the intersection product of two divisors, defined on the Picard group $\Pic(X)$ of the surface $X$ \cite[Chapter 1]{Beauville:ComplexAlgebraicSurfaces}.
Then, by plugging such quantities in \cite[Theorem 10.1]{Kazarian:Multisingularities} one gets the desired numbers. The linear system $|D|$ is given by the embedding of $X$ as the linear system of hyperplane sections. We explain the notation $k \pi^*(L) - \sum E_i$ below.
Table \ref{table:divisors_parameters} computes \eqref{eq:Kaz_quantities} for our surfaces. Using these numbers one can get the degrees in Table \ref{table:degreesXk}:
\begin{equation}
    \deg X^{[k]} = N_{A_1^k},
\end{equation}
using the notation of \cite[Theorem 10.1]{Kazarian:Multisingularities} in which $N_{A_1^k}$ is the number of singular curves of $|D|$ with $k$ nodes.
\begin{table}[ht!]
    \renewcommand{\arraystretch}{1.7}
    \centering
    \begin{tabular}{c||c|c|c|c|c|c|}
    \textbf{surface} & $D$ & $K_X$ & $\mathsf{d}$ & $\mathsf{k}$ & $\mathsf{s}$ & $\mathsf{x}$ \\
    \hline
    \hline
    Veronese & $2L$ & $-3L$ & $4$ & $-6$ & $9$ & $3$ \\
    \hline
    Del Pezzo & $3\pi^*(L)-\sum_{i=1}^5 E_i$ & $-3\pi^*(L)+\sum_{i=1}^5 E_i$ & $4$ & $-4$ & $4$ & $8$ \\
    \hline
    Bordiga & $4\pi^*(L)-\sum_{i=1}^{10} E_i$ & $-3\pi^*(L)+\sum_{i=1}^{10} E_i$ & $6$ & $-2$ & $-1$ & $13$ \\
    \hline
    K3 & $H$ & $0$ & $6$ & $0$ & $0$ & $24$ \\
    \hline
    \end{tabular}
    \caption{The quantities required for Kazarian's formulae. Here $H$ is the class of a hyperplane in $\PP^4$, which identifies a smooth curve of genus $4$.}
    \label{table:divisors_parameters}
\end{table}     

\section{Veronese}\label{sec:veronese}

Consider the Veronese surface $Y$ in $\PP^5$, i.e., the image of the second Veronese map
\begin{align}
    \nu_2 : \PP^2 &\to \PP^5 \\
    [ t_0, t_1, t_2] &\mapsto [ t_0^2,  t_0  t_1,  t_0  t_2,  t_1^2,  t_1  t_2,  t_2^2].
\end{align}
Its secant variety is a proper hypersurface in $\PP^5$ \cite{AleHir:PolynomialInterpolation}. Therefore, by projecting $Y$ onto $\PP^4$ from a point that does not belong to the secant variety, we obtain a smooth irreducible surface $X\subset \PP^4$. We only consider the real surfaces obtained as the projection from a real point in $\PP^5$ and no other real forms of this complex algebraic surface. We now make this more precise.

\subsection[The complex picture]{The $\C$omplex picture}

We treat $\PP^5$ as the projective space of the space of $3\times 3$ symmetric matrices. Using this point of view, one can prove that the Veronese surface $Y$ is the variety of rank one matrices. We will first study $Y$ and its associated varieties $Y^{[k]}$, and then we will take care of the projection. In this setting, we have the pairing of two matrices consisting of the trace of their product; we will generally denote by $A$ matrices in the primal space and by $B$ matrices in the dual space. Given a symmetric matrix $B\in (\PP^5)^*$, the corresponding hyperplane in the primal space is
\begin{equation}
    H_B = \{ A \in \PP^5 \,|\, \trace (BA) = 0 \}.
\end{equation}
On the other hand, every $3\times 3$ symmetric matrix $B$ corresponds to the plane conic $C_B = \{ t\in \PP^2 \,|\,  t^T B  t = 0\}\subset \PP^2$. This plane curve is isomorphic via $\nu_2$ to the hyperplane section $H_B\cap Y$, since $ t^T B  t = \trace (B ( t  t^T))$ and clearly $ t  t^T \in Y$. So in order to study hyperplane sections of the Veronese surface, we can study plane conics. There are three distinct cases: if $\rank B = 3$, then $C_B$ is a smooth conic; if $\rank B = 2$, then $C_B$ is the union of two lines;
if $\rank B = 1$, then $C_B$ is a double line. Hence, the number of singular points of $C_B$, and therefore of the corresponding hyperplane section of the Veronese, can be $0$, $1$, or infinity. In the last case, the hyperplane section is a double conic. This shows that
\begin{equation}
    Y^{[2]} = Y^{[3]} = \{ B\in (\PP^5)^* \,|\, \rank B = 1\} = \nu_2((\PP^2)^*)
\end{equation}
is again a Veronese surface (in the dual space). Then, $\deg \left(Y^{[2]}\right)^* = 6$. Since a conic is always planar, $Y^{[4]} = \emptyset$.

We now focus on the projection onto $\PP^4$. The secant variety of $Y$ is the set of $3\times 3$ symmetric matrices of rank $2$. Thus, fix a matrix $A \in \PP^5$ of full rank and let $\pi_A : \PP^5 \dashrightarrow \PP^4$ be the projection centered at $A$. Denote by $X_A = \pi_A (Y)$ our surface of interest in $\PP^4$. The projection $\pi_A$ corresponds, in the dual space, to the intersection with the hyperplane $H_A$. Therefore, the variety $X_A^{[2]} = X_A^{[3]}$ is given by $Y^{[2]} \cap H_A = \nu_2( (\PP^2)^* ) \cap H_A$. By the reasoning above, it is isomorphic to the plane conic $C_A \subset (\PP^2)^*$. Since $A$ is full rank, the variety $\nu_2(C_A) \subset (\PP^4)^*$ is the rational normal curve. Therefore,
\begin{equation}
    \dim X_A^{[2]} = 1, \qquad \deg X_A^{[2]} = 4.
\end{equation}
The dual variety $\left( X_A^{[2]} \right)^* \!\!\!=\! \left( X_A^{[3]} \right)^* \!\!\!\subset \!\PP^4$ is a sextic hypersurface.
This completes the analysis of the complex structure.

\subsection[The real picture]{The $\R$eal picture}
Regarding the real parts of these varieties and the possible contribution to the algebraic boundary of $\conv X_A$, we need to distinguish two cases in terms of the signature of the projection center $A$. Since $A$ must be of rank $3$ in order for $X_A\subset \PP^4$ to be smooth, its signature can be either $(3,0)$ or $(2,1)$, up to a global sign. Here $(p,n)$ specifies the number $p$ of positive eigenvalues of $A$ and the number $n$ of negative eigenvalues.

Firstly, we assume that $A$ has signature $(3,0)$, i.e., it is positive definite. Pick any point $\overline{M}\in \PP^4$; then $\pi_A^{-1}(\overline{M}) = \{\mu M + \lambda A | [\mu,\lambda]\in \PP^1\}$. For $\mu = 1$ and a sufficiently large real $\lambda$, the matrix $M+\lambda A$ is positive definite as well, and therefore it is a combination of three rank one matrices: $M+\lambda A = \frac{1}{3} (v_1 v_1^T + v_2 v_2^T + v_3 v_3^T)$. Hence $\overline{M}\in \conv X_A$, i.e., $\conv X_A = \R^4$. In this case, the conic $C_A$ and the rational normal curve $X_A^{[2]}$ are real varieties with empty real locus. The real points of $\left( X_A^{[2]} \right)^*$ must then be singular, so they do not form a hypersurface. Therefore, there are no hypersurfaces that can describe $\partial_a \conv X_A$. The geometric picture in this case is that $\conv X_A $ is the a projection of the cone of positive semidefinite matrices from an interior point and therefore $\R^4$. 

Assume now that the center of the projection $A$ has signature $(2,1)$. In this case the rational normal curve $X_A^{[2]}$ has a non-empty real locus, and its dual hypersurface has real smooth points. Geometrically, $\conv X$ is the projection of the cone of positive semidefinite matrices, from a point outside it. We summarize our findings in the following theorem.

\begin{theorem}\label{thm:veronese}
Let $X_A\subset \PP^4$ be the projection of $\nu_2(\PP^2)\subset \PP^5$ with center the full-rank $3\times 3$ symmetric matrix $A$. Then,
\begin{itemize}
    \item if $A$ has signature $(3,0)$, then $\conv X_A = \R^4$;
    \item if $A$ has signature $(2,1)$, then $\partial_a (\conv X_A)$ is a hypersurface of degree $6$.
\end{itemize}
\end{theorem}

\begin{example}
As an explicit example of the $A=(2,1)$ case, consider the projection of the Veronese surface onto $\PP^4$ centered at $A = [0,0,-1,2,0,0]\in\PP^5$, which corresponds to the matrix $$
\begin{bmatrix}
0 & 0 & -\frac{1}{2} \\
0 & 2 & 0 \\
-\frac{1}{2} & 0 & 0
\end{bmatrix}$$
having signature $(2,1)$. Then, $X_A$ is the image of the map
\begin{align}
    \nu_2 : \PP^2 &\to \PP^4 \\
    [t_0,t_1,t_2] &\mapsto [t_0^2, -2 t_0 t_1, 2 t_0 t_2 + t_1^2, -2 t_1 t_2, t_2^2].
\end{align}
Its ideal is generated by the following $7$ cubic polynomials:
\begin{gather}
    x_3^3-4 x_2 x_3 x_4+8 x_1 x_4^2,\\
    x_2 x_3^2-4 x_2^2 x_4+2 x_1 x_3 x_4+16 x_0 x_4^2,\\
    x_1 x_3^2-4 x_1 x_2 x_4+8 x_0 x_3 x_4,\\
    x_0 x_3^2-x_1^2 x_4,\\
    x_1^2 x_3-4 x_0 x_2 x_3+8 x_0 x_1 x_4,\\
    x_1^2 x_2-4 x_0 x_2^2+2 x_0 x_1 x_3+16 x_0^2 x_4,\\
    x_1^3-4 x_0 x_1 x_2+8 x_0^2 x_3.
\end{gather}
We identify $\R^4\subset \PP^4(\R)$ with the affine chart whose hyperplane at infinity is given by the equation $2 x_0 + x_1 + 2 x_2 + 4 x_3 = 0$. Since 
\begin{equation}
    2 t_0^2 - 2 t_0 t_1 + 2 ( 2 t_0 t_2 + t_1^2 ) + 4 t_2^2 = ( t_0 - t_1 )^2 + (t_0 + 2 t_2)^2 + t_1^2
\end{equation}
is a sum of three squares, then the Veronese surface in this affine chart does not have real points at infinity, hence it is compact. This implies that its convex hull is a convex body and by Theorem \ref{thm:veronese} its algebraic boundary is the variety $\left( X_A^{[2]} \right)^*$ of degree $6$ defined as the zero locus of
\begin{gather}
    x_1^2 x_2^2 x_3^2-4 x_0 x_2^3 x_3^2-4 x_1^3 x_3^3+18 x_0 x_1 x_2 x_3^3-27 x_0^2 x_3^4-4 x_1^2 x_2^3 x_4\\
    +16 x_0 x_2^4 x_4+18 x_1^3 x_2 x_3 x_4-80 x_0 x_1 x_2^2 x_3 x_4-6 x_0 x_1^2 x_3^2 x_4+144 x_0^2 x_2 x_3^2 x_4\\
    -27 x_1^4 x_4^2+144 x_0 x_1^2 x_2 x_4^2-128 x_0^2 x_2^2 x_4^2-192 x_0^2 x_1 x_3 x_4^2+256 x_0^3 x_4^3
\end{gather}
which is dual to the standard quartic rational normal curve in $(\PP^4)^*$ parametrized by the map $[s_0,s_1]\mapsto[s_0^4,s_0^3s_1,s_0^2s_1^2,s_0s_1^3,s_1^4]$, for $[s_0,s_1]\in\PP^1$.
\end{example}

\begin{example}
For the $A=(3,0)$ case, let us pick the identity matrix, namely the point $[1,0,0,1,0,1]\in\PP^5$, and consider the projection of the Veronese surface given by
\begin{align}
    \nu_2 : \PP^2 &\to \PP^4 \\
    [t_0,t_1,t_2] &\mapsto [t_0^2 - t_1^2, t_0^2 - t_2^2, t_0 t_1, t_0t_2, t_1t_2].
\end{align}
The image of $\nu_2$ is our surface of interest $X_A$, namely the projection of the standard Veronese surface to $\PP^4$. The ideal of $X_A$ is generated by 
\begin{gather}
    x_2 x_3^2-x_0 x_3 x_4-x_2 x_4^2,\\
    x_2^2 x_3-x_1 x_2 x_4-x_3 x_4^2,\\
    x_1 x_2 x_3-x_0 x_1 x_4-x_2^2 x_4+x_4^3,\\
    x_0 x_2 x_3-x_0 x_1 x_4-x_3^2 x_4+x_4^3,\\
    x_0^2 x_3-x_0 x_1 x_3-x_3^3+x_0 x_2 x_4+x_3 x_4^2,\\
    x_0 x_1 x_2-x_1^2 x_2+x_2^3-x_1 x_3 x_4-x_2 x_4^2,\\
    x_0^2 x_1-x_0 x_1^2+x_0 x_2^2-x_1 x_3^2-x_0 x_4^2+x_1 x_4^2.
\end{gather}
The curve $X_A^{[2]}$ is a rational normal quartic with no real points. It is the zero locus of the ideal with generators
\begin{gather}
    x_2 x_3-2 x_0 x_4-2 x_1 x_4,\\
    2 x_0 x_3+x_2 x_4,\\
    2 x_1 x_2+x_3 x_4,\\
    4 x_1^2+x_3^2+x_4^2,\\
    4 x_0 x_1-x_4^2,\\
    4 x_0^2+x_2^2+x_4^2.
\end{gather}
The corresponding dual hypersurface $\left( X_A^{[2]} \right)^*$, is defined by the sextic:
\begin{align}
    s(x_0,\ldots,x_4) =& \; x_0^4 x_1^2-2 x_0^3 x_1^3+x_0^2 x_1^4+2 x_0^3 x_1 x_2^2+2 x_0^2 x_1^2 x_2^2-8 x_0 x_1^3 x_2^2+4 x_1^4 x_2^2+x_0^2 x_2^4 \\ &+8 x_0 x_1 x_2^4
    -8 x_1^2 x_2^4+4 x_2^6+4 x_0^4 x_3^2-8 x_0^3 x_1 x_3^2+2 x_0^2 x_1^2 x_3^2+2 x_0 x_1^3 x_3^2\\ 
    &+20 x_0^2 x_2^2 x_3^2 -38 x_0 x_1 x_2^2 x_3^2
    +20 x_1^2 x_2^2 x_3^2+12 x_2^4 x_3^2-8 x_0^2 x_3^4+8 x_0 x_1 x_3^4+x_1^2 x_3^4\\
    &+12 x_2^2 x_3^4+4 x_3^6+8 x_0^3 x_2 x_3 x_4-12 x_0^2 x_1 x_2 x_3 x_4-12 x_0 x_1^2 x_2 x_3 x_4+8 x_1^3 x_2 x_3 x_4\\
    &+36 x_0 x_2^3 x_3 x_4-72 x_1 x_2^3 x_3 x_4-72 x_0 x_2 x_3^3 x_4
    +36 x_1 x_2 x_3^3 x_4-2 x_0^3 x_1 x_4^2\\
    &+8 x_0^2 x_1^2 x_4^2-2 x_0 x_1^3 x_4^2+2 x_0^2 x_2^2 x_4^2-2 x_0 x_1 x_2^2 x_4^2+20 x_1^2 x_2^2 x_4^2
    +12 x_2^4 x_4^2\\
    &+20 x_0^2 x_3^2 x_4^2-2 x_0 x_1 x_3^2 x_4^2+2 x_1^2 x_3^2 x_4^2-84 x_2^2 x_3^2 x_4^2+12 x_3^4 x_4^2+36 x_0 x_2 x_3 x_4^3
    \\
    &+36 x_1 x_2 x_3 x_4^3+x_0^2 x_4^4-10 x_0 x_1 x_4^4+x_1^2 x_4^4+12 x_2^2 x_4^4+12 x_3^2 x_4^4+4 x_4^6.
\end{align}
A computation in \texttt{Julia} shows that this polynomial is a sum of squares. For instance:
\begin{align}
    s(x_0,\ldots,x_4) = &\;\, 4 (-x_1^2 x_2+x_0 x_1 x_2-2 x_2 x_3^2+x_2^3-x_1 x_3 x_4+2 x_0 x_3 x_4+x_2 x_4^2)^2 \\
    &+ 4 (x_0^2 x_3-x_0 x_1 x_3+2 x_2^2 x_3-x_3^3+x_0 x_2 x_4-2 x_1 x_2 x_4-x_3 x_4^2)^2 \\
    &+ (2 x_0 x_1 x_4 - x_0 x_2 x_3 - x_1 x_2 x_3 + x_2^2 x_4 + x_3^2 x_4 - 2 x_4^3)^2 \\
    &+ (x_0^2 x_1 - x_0 x_1^2 + x_0 x_2^2 - x_0 x_4^2 - x_1 x_3^2 + x_1 x_4^2)^2 \\
    &+ 3 (x_0 x_2 x_3-x_1 x_2 x_3+x_2^2 x_4-x_3^2 x_4)^2 \\
    &+ 12 (x_2 x_3^2-x_0 x_3 x_4-x_2 x_4^2)^2 \\
    &+ 12 (x_2^2 x_3-x_1 x_2 x_4-x_3 x_4^2)^2. 
\end{align}
The seven cubic equations in this sum-of-squares decomposition generate the ideal of the Veronese surface $X_A \in \PP^4$. This surface has smooth real points and is the (Zariski closure of the) real zero locus of $s$. Therefore, the real points of $\left( X_A^{[2]} \right)^*$ do not form a hypersurface. Moreover, it is one of the two irreducible components of the singular locus of the sextic hypersurface defined by $s$ in $\PP^4$. 
\end{example}

\section{Del Pezzo}\label{sec:del_pezzo}
In this section we study the case in which $X\subset \PP^4$ is a Del Pezzo surface of degree $4$. Over $\C$ this type of surfaces arises as a complete intersection of two quadratic threefolds. Alternatively, $X$ can be realized as a complex algebraic surface as the blow up $\mathrm{Bl}_{p_1,\ldots,p_5} \PP^2$ of five points in the (complex) projective plane, such that no three points lie on a line \cite[Proposition 8.1.25]{Dolgachev:ClassicalAlgebraicGeometry}. Such a surface can then be embedded in $\PP^4$ via its anticanonical linear system $D = 3\pi^*(L) - \sum_{i=1}^5 E_i$, where $L$ is the class of a line in $\PP^2$, $\pi^*(L)$ is its pullback on $\mathrm{Bl}_{p_1,\ldots,p_5} \PP^2$, and $E_i$ is the exceptional divisor corresponding to $p_i$ (see \cite[Chapter 8]{Dolgachev:ClassicalAlgebraicGeometry} for more).
We study the real picture for all possible real forms of Del Pezzo surfaces of degree $4$ in subsection~\ref{subsec:delPezzoReal} below. The real forms have been classified for instance in \cite{russo:delPezzo}.

\subsection[The complex picture]{The $\C$omplex picture}

The blow up construction will be used also in the next sections, so we explain the idea in more detail. The space of plane cubics through five generic points is a $4$-dimensional projective space. We fix a basis $\{q_0,\ldots,q_4\}\subset \C[x_0,x_1,x_2]_{3}$ for such a space. Then $D$, i.e., the class of the strict transform of a plane cubic passing simply by $p_1,\ldots,p_5$, induces the following map:
\begin{align}
\varphi : \PP^2 &\dashrightarrow \PP^4 \\
x &\mapsto [q_0(x), \ldots , q_4(x)].
\end{align}
The base locus coincides with the five points $p_1,\ldots,p_5$. The Del Pezzo surface $X$ is the closure of the image of $\varphi$. Hence, $\varphi$ is birational to $X$. Let us now consider the blow up of $\PP^2$ in the $p_i$'s. We have the following:
\begin{equation}
\begin{tikzcd}
\textrm{Bl}_{p_1, \ldots ,p_5} \PP^2 \arrow[dr, "\widetilde{\varphi}"] \arrow[d, "\pi"] & \\
\PP^2 \arrow[r, dashed, "\varphi"] & X \subset \PP^4
\end{tikzcd}
\end{equation}
Here $\widetilde{\varphi}$ is an isomorphism between the blow up and $X$. Notice that the image of an exceptional divisor $E_i \subset \textrm{Bl}_{p_1, \ldots ,p_5} \PP^2$ is a line in $\PP^4$, that with some abuse of notation we will keep calling $E_i$.
Consider a hyperplane $u^\perp$ in $\PP^4$, identified by a vector $u$, namely
\begin{equation}
    u^\perp = \{ [x_0,\ldots ,x_4] \in \PP^4 \, |\, u_0 x_0 + \ldots + u_4 x_4 = 0 \}.
\end{equation}
Such a hyperplane identifies two curves in two different spaces. On the one hand, we have the hyperplane section $C_u = u^\perp \cap X$ in $\PP^4$, a curve in a hyperplane of $\PP^4$. On the other hand, we have the cubic $Q_u = \{ y\in \PP^2 \,|\, q_u(y) = u_0 q_0(y) + \ldots + u_4 q_4(y) = 0\}$ in $\PP^2$ which is the pullback of the hyperplane section via $\varphi$. We denote the class of $C_u$ (as a divisor on $X$) by $u^\perp . X$. 
Then, if $\widehat{Q}_u$ is the strict transform of $Q_u$, the following relation in $\Pic(\textrm{Bl}_{p_1, \ldots ,p_5} \PP^2)$ holds:
\begin{equation}\label{eq:multiplicities_singularities}
    \widehat{Q}_u + \sum_{i=1}^{5} (\textrm{mult}_{p_i} Q_u - 1) E_i =  \widetilde{\varphi}^*(u^\perp . X).
\end{equation}
Informally, this means that we can compare singularities of $C_u$ and $Q_u$. Let us denote for simplicity $\widetilde{Q}_u = \widehat{Q}_u + \sum_{i=1}^{5} (\textrm{mult}_{p_i} Q_u - 1) E_i$. Notice that $E_i \subset C_u$ if and only if $\textrm{mult}_{p_i} Q_u \geq 2$. For instance, $E_1 \subset C_u$, i.e., $E_1 \subset u^\perp$, if and only if $Q_u$ is singular at $p_1$.
Therefore, instead of studying hyperplanes tangent to $X$ at a certain number of points, we can focus on curves in $\PP^2$ passing through $p_1,\ldots,p_5$, with prescribed singularities.
Using this point of view, we want to analyze the irreducible components of the varieties $X^{[k]}\subset (\PP^4)^*$ whose duals are hypersurfaces. First, we discuss the various irreducible components. The findings in the cases $k=2,3,4$ are summarized in Propositions \ref{prop:DelPezzo_2}, \ref{prop:DelPezzo_3}, \ref{prop:DelPezzo_4}. 

Since a node at a base point corresponds to two singularities of $C_u$, whereas a node outside the base locus gives one singularity of $C_u$, our following case analysis for $X^{[k]}$ is based on partitions of $k$ into parts that are restricted to be $1$ or $2$.

\subsubsection{The irreducible components for k=2}
Points of $X^{[2]}\subset (\PP^4)^*$ are hyperplanes in $\PP^4$ tangent to $X$ at $2$ points. Every such $u\in X^{[2]}$ identifies a plane cubic curve $Q_u$ as explained above, such that $\widetilde{Q}_u$ has $2$ singular points. We will see that the family of these curves is two-dimensional, as expected, so that $X^{[2]}$ is a surface. There are four configurations of plane cubics through $p_1,\ldots,p_5$ in which $\widetilde{Q}_u$ has $2$ nodes.
The only partitions of $2$ are $2 = 1\cdot 2$ and $2 = 2\cdot 1$. The first partition (meaning a node at a base point) leads to case (A), whereas the second partition (meaning two nodes in $\PP^2$ outside of the base locus) results in cases (B), (C), and (D).
We describe them in detail in the following paragraphs. We point out that by \emph{conic} we always mean a smooth quadric in $\PP^2$.

\paragraph{(A): irreducible cubic.}
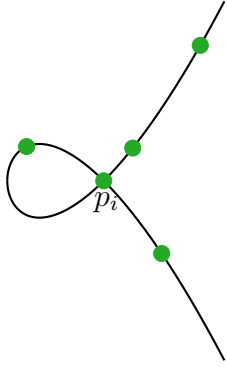
\begin{figure}[ht]
    \centering
    \begin{tikzpicture}[scale=1]
\begin{axis}[
width=2in,
height=2in,
scale only axis,
xmin=-2,
xmax=2,
ymin=-2,
ymax=2,
ticks = none, 
ticks = none,
axis background/.style={fill=white},
axis line style={draw=none} 
]

\draw[thick,variable=\t,domain=0:1.5,samples=50]
  plot ({\t^2-1},{\t*(\t^2-1)});
  
\draw[thick,variable=\t,domain=0:1.5,samples=50]
  plot ({\t^2-1},{-\t*(\t^2-1)});

\addplot[only marks,mark=*,mark size=3pt,mygreen,
]  coordinates {
    (0,0) (1,1.4142) (0.3,0.3420) (0.6,-0.7589) (-0.8,0.3577) 
};

\node[black] (P) at (axis cs:0.03,-0.22) {$p_i$};

\end{axis}
\end{tikzpicture}
    \caption{Irreducible cubic curve through five fixed points, with a node at $p_i$.}
    \label{fig:X2_caseA}
\end{figure}
Requiring that a plane cubic curve has two nodes implies that it is reducible. Otherwise, the intersection of the cubic with the line through the two nodes would contradict B\'ezout's Theorem. Therefore, the only way to get two singularities on $\widetilde{Q}_u$ from an irreducible cubic $Q_u$ is to put a node at one of the $p_i$'s, see Figure \ref{fig:X2_caseA}. This implies that $C_u$, up to isomorphism, is the union of $\widehat{Q}_u$ and one of the exceptional lines. Having a node at a fixed point is a linear condition on the coefficients of the cubic.
In each of the five cases, corresponding to the choice of one of the $p_i$'s, we get a $\PP^2\subset (\PP^4)^*$ of such cubics. Therefore, the contribution of these irreducible components to the degree of $X^{[2]}$ is $1$ each and so $5\cdot 1 = 5$ in total.

\paragraph{(B): conic $+$ line.}
The first possible way for a reducible plane cubic to have two nodes, is to force the conic to go through the five points and take any line, see Figure \ref{fig:X2_caseBCD}, left. There is just one way to do so, since there is only one conic through five points. Hence, the irreducible component of $X^{[2]}$ whose points identify this type of plane cubics, is a $\PP^2\subset (\PP^4)^*$ and therefore of degree $1$.
\begin{figure}[ht]
    \centering
    \begin{tikzpicture}[scale=1]
\begin{axis}[
width=2in,
height=2in,
scale only axis,
xmin=-2,
xmax=2,
ymin=-2,
ymax=2,
ticks = none, 
ticks = none,
axis background/.style={fill=white},
axis line style={draw=none} 
]

\draw[thick,rotate around={-20:(0,0)}] (0,0) ellipse (1.5 and 1);
  
\draw[thick] (-2,-1) -- (1.3,1.3);

\addplot[only marks,mark=*,mark size=3pt,mygreen,
]  coordinates {
    (0.2,0.98) (1.44,-0.2) (0.9,0.63) (-0.5,-0.86) (-1.44,0.3577)
};

\end{axis}
\end{tikzpicture}
    \begin{tikzpicture}[scale=1]
\begin{axis}[
width=2in,
height=2in,
scale only axis,
xmin=-2,
xmax=2,
ymin=-2,
ymax=2,
ticks = none, 
ticks = none,
axis background/.style={fill=white},
axis line style={draw=none} 
]

\draw[thick,rotate around={-20:(0,0)}] (0,0) ellipse (1.5 and 1);
  
\draw[thick] (-2,-1) -- (1.3,1.3);

\addplot[only marks,mark=*,mark size=3pt,mygreen,
]  coordinates {
    (0.2,0.98) (1.44,-0.2) (-1.44,0.3577)
    (0,0.3939) (-1.5,-0.6515)
};

\end{axis}
\end{tikzpicture}
    \begin{tikzpicture}[scale=1]
\begin{axis}[
width=2in,
height=2in,
scale only axis,
xmin=-2,
xmax=2,
ymin=-2,
ymax=2,
ticks = none, 
ticks = none,
axis background/.style={fill=white},
axis line style={draw=none} 
]

\draw[thick,rotate around={-20:(0,0)}] (0,0) ellipse (1.5 and 1);
  
\draw[thick] (-2,-1) -- (1.3,1.3);

\addplot[only marks,mark=*,mark size=3pt,mygreen,
]  coordinates {
    (0.2,0.98) (1.44,-0.2) (-0.5,-0.86) (-1.44,0.3577)
    (0,0.3939)
};

\end{axis}
\end{tikzpicture}
    \caption{Left: a conic through five points and a line. Center: a conic through three points and a line through two points. Right: a conic through four points and a line through one point.}
    \label{fig:X2_caseBCD}
\end{figure}
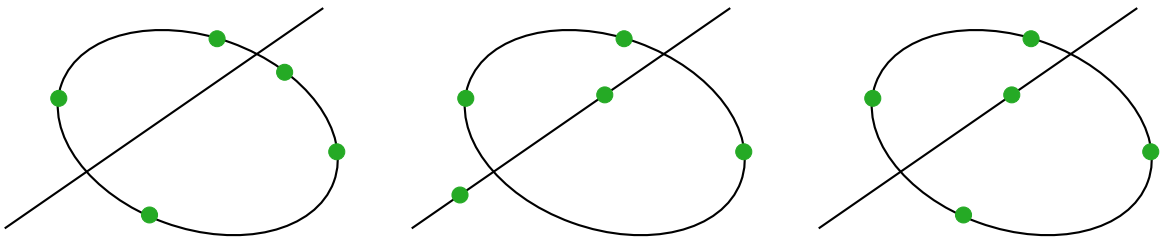

\paragraph{(C): conic $+$ line.}
Another possibility is to fix the line instead of the conic. Choose two of the five points (this can be done in $\binom{5}{2}=10$ ways), force the line to go through them, and take any conic through the remaining three points (Figure \ref{fig:X2_caseBCD}, center). Again, for every choice of the two points, the associated component of $X^{[2]}$ is a $\PP^2 \subset (\PP^4)^*$, hence the sum of their degrees is $10\cdot 1 = 10$.

\paragraph{(D): conic $+$ line.}
The last possible case is obtained by imposing that the conic goes through four points and the line through one, as in Figure \ref{fig:X2_caseBCD}, right. This can be done in $5$ different ways, and every time we get a copy of $\PP^1\times\PP^1$. Therefore, the degree of this type of  components is $5\cdot 2 = 10$.

This case distinction is exhaustive: there are no other ways for $\widetilde{Q}_u$ to have two singularities. In fact, if $Q_u$ is irreducible, then it is forced to have just one singularity by B\'ezout's Theorem, and this is our case (A). On the other hand, if $Q_u$ is reducible, it can be either the union of a conic and a line, or the union of three lines. In the first case, the intersection of conic and line makes two nodes, as required. Since the union of the irreducible components has to contain the $5$ points $p_1,\ldots,p_5$ and the line can contain at most $2$, the conic can contain $5$, $4$, or $3$ of them. These are precisely the cases (B), (D), (C). Finally, if $Q_u$ is the union of three lines, then it must have at least three nodes. Therefore this case does not contribute to $X^{[2]}$.
Summing up over all the irreducible components that we have just described, we get that $\deg X^{[2]} = 5+1+10+10 = 26$ as predicted in Table \ref{table:degreesXk}.

We now determine for which of these irreducible components, the dual varieties are hypersurfaces: this does not happen when the component is $\PP^2\subset (\PP^4)^*$. Therefore, the only relevant case is (D).
The dual of a smooth quadric surface, for us $\PP^1\times \PP^1$ naturally embedded in a hyperplane of $(\PP^4)^*$, has degree $2$ \cite[Proposition 2.9]{EisHar:3264}: it is a quadratic cone with vertex the hyperplane that contains $\PP^1\times \PP^1$. Hence, by dualizing these $\PP^1\times \PP^1$, we get the union of the five singular quadrics of the pencil defined by $X$. 
More precisely, $X$ being the complete intersection of two quadratic hypersurfaces $V_0,V_\infty \subset \PP^4$, we can construct the associated pencil of quadrics that they generate:
\begin{equation}
    \mathscr{L} = \{\lambda V_0 + \mu V_\infty \,|\, [\lambda,\mu] \in \PP^1 \}.
\end{equation}
Every element of $\mathscr{L}$ is a quadric in $\PP^4$, which is therefore represented by a $5\times 5$ matrix with linear homogeneous entries in $\lambda, \mu$. The determinant of this matrix has five zeros, which are isolated exactly when $X$ is smooth. These values of $[\lambda,\mu]$ define the five singular quadrics $V_1,\ldots,V_5$ of the pencil $\mathscr{L}$.
To see that these constitute $\left( X^{[2]}_{(D)} \right)^*$, note that $\PP^1\times\PP^1$ is contained in a $\PP^3$ inside $(\PP^4)^*$. Hence, its dual variety is a cone of degree $2$ containing $X$: this is by definition one of the $V_i$'s. 
We summarize our findings on $X^{[2]}$ and $\left( X^{[2]} \right)^*$.

\begin{proposition}\label{prop:DelPezzo_2}
Let $X\subset\PP^4$ be a smooth Del Pezzo surface, arising as the complete intersection of two generic quadratic hypersurfaces $V_0, V_\infty$. Then, $X^{[2]}\subset (\PP^4)^*$ is the union of $21$ irreducible components, $16$ of which are copies of $\PP^2$. The remaining $5$ are the singular quadrics in the pencil generated by $V_0, V_{\infty}$. Therefore, $\deg X^{[2]} = 26$.\\
The dual variety $\left( X^{[2]} \right)^*\subset \PP^4$ is the union of $16$ copies of $\PP^1$ and $5$ quadratic hypersurfaces, namely the singular quadrics of the pencil. 
\end{proposition}

\subsubsection{The irreducible components for k=3}
The points of $X^{[3]}\subset (\PP^4)^*$ are hyperplanes in $\PP^4$ tangent to $X$ at $3$ points. Every such $u\in X^{[3]}$ identifies a plane cubic curve $Q_u$ such that $\widetilde{Q}_u$ has $3$ singular points. We will see that the family of these curves is one-dimensional, as expected, so that $X^{[3]}$ is a curve.
We distinguish two cases in which a plane cubic curve $Q_u$ gives $\widetilde{Q}_u$ with $3$ singular points.
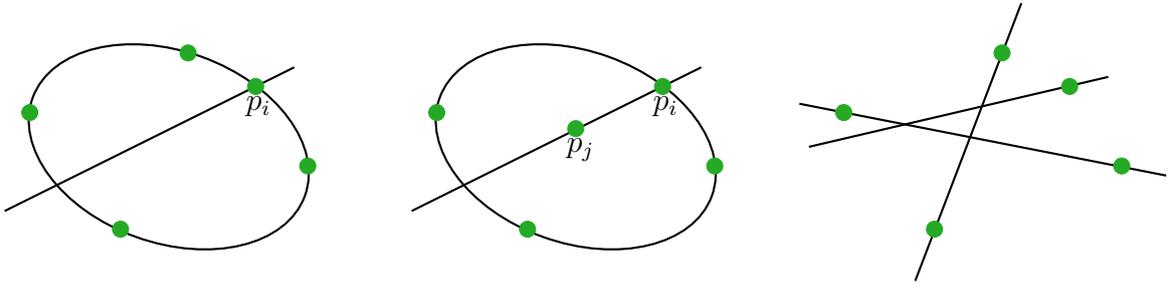
\begin{figure}[ht]
    \centering
    \begin{tikzpicture}[scale=1]
\begin{axis}[
width=2in,
height=2in,
scale only axis,
xmin=-2,
xmax=2,
ymin=-2,
ymax=2,
ticks = none, 
ticks = none,
axis background/.style={fill=white},
axis line style={draw=none} 
]

\draw[thick,rotate around={-20:(0,0)}] (0,0) ellipse (1.5 and 1);
  
\draw[thick] (-1.7,-0.67) -- (1.3,0.83);

\addplot[only marks,mark=*,mark size=3pt,mygreen,
]  coordinates {
    (0.2,0.98) (1.44,-0.2) (0.9,0.63) (-0.5,-0.86) (-1.44,0.3577)
};

\node[black] (P) at (axis cs:0.93,0.42) {$p_i$};

\end{axis}
\end{tikzpicture}
    \begin{tikzpicture}[scale=1]
\begin{axis}[
width=2in,
height=2in,
scale only axis,
xmin=-2,
xmax=2,
ymin=-2,
ymax=2,
ticks = none, 
ticks = none,
axis background/.style={fill=white},
axis line style={draw=none} 
]

\draw[thick,rotate around={-20:(0,0)}] (0,0) ellipse (1.5 and 1);
  
\draw[thick] (-1.7,-0.67) -- (1.3,0.83);

\addplot[only marks,mark=*,mark size=3pt,mygreen,
]  coordinates {
    (1.44,-0.2) (0.9,0.63) (-0.5,-0.86) (-1.44,0.3577)
    (0,0.19) 
};

\node[black] (P) at (axis cs:0.93,0.42) {$p_i$};
\node[black] (P) at (axis cs:0.05,-0.03) {$p_j$};

\end{axis}
\end{tikzpicture}
    \begin{tikzpicture}[scale=1]
\begin{axis}[
width=2in,
height=2in,
scale only axis,
xmin=-2,
xmax=2,
ymin=-2,
ymax=2,
ticks = none, 
ticks = none,
axis background/.style={fill=white},
axis line style={draw=none} 
]
  
\draw[thick] (0.4,1.5) -- (-0.7,-1.4);
\draw[thick] (1.9,-0.3) -- (-1.9,0.45);
\draw[thick] (-1.8,0) -- (1.3,0.73);

\addplot[only marks,mark=*,mark size=3pt,mygreen,
]  coordinates {
    (0.2,0.98) (1.44,-0.2) (0.9,0.63) (-0.5,-0.86) (-1.44,0.3577)
};

\end{axis}
\end{tikzpicture}
    \caption{Left: a conic through five points and a line through one point: they intersect at $p_i$. Center: a conic through four points and a line through two points: they intersect at $p_i$. Right: three lines through five points.}
    \label{fig:X3}
\end{figure}

\noindent
The possible partitions of $3$ are $3 = 1\cdot 2 + 1\cdot 1$ and $3 = 3\cdot 1$. The first partition imposes a node at a base point (case (A)). The other partition is realized in case (B).

\paragraph{(A): a node at $p_i$.}
The first possibility is to impose that $Q_u$ has a node at one of the $p_i$'s. This makes two singularities in $\widetilde{Q}_u$. So we need that $Q_u$ has one additional node somewhere in $\PP^2$. This makes it reducible, as in Figure \ref{fig:X3}, left or center. The curve of $u\in (\PP^4)^*$, that identify plane curves $Q_u$ of this type, has degree $5$ for any choice of the node $p_i$.
The curve in $(\PP^4)^*$ is indeed the union of five lines $\PP^1\subset (\PP^4)^*$. One of them arises by fixing the conic through $p_1,\ldots,p_5$ and moving the line through $p_i$; the other four arise by fixing the line through $p_i$ and another point $p_j$ ($j\neq i)$.
In total, the degree of these components sums to $5\cdot 5 = 25$. To compute the degree $5$ we can also use Kazarian's formulae for the case 
\begin{equation}
    D = 3\pi^*(L) - 2E_i - \sum_{k\neq i} E_k \qquad \quad K_X = -3\pi^*(L) + \sum_{i=1}^5 E_i
\end{equation}
which implies, in Kazarian's notation, that $\mathsf{d} = 1, \mathsf{k} = -3, \mathsf{s} = 4, \mathsf{x} = 8$. Substituting these quantities in the expression for $N_{A_1}$ gives $5$.

\paragraph{(B): three lines.}
The other possibility is that the cubic splits into three lines. We fix two pairs of base points and take the two lines through them; the third line goes through the last point (see Figure \ref{fig:X3}, right). There are $5\cdot \frac{1}{2}\binom{4}{2} = 15$ ways in which one can choose the roles of the five points on the three lines. Each choice determines a line $\PP^1\subset X^{[3]}$. So these components contribute with degree $15 \cdot 1 = 15$.

These cases exhaust all the possibilities for $k=3$. In fact, the plane cubic $Q_u$ must be reducible since it needs to be singular in at least two points in $\PP^2$. The union of a conic and a line makes two singularities in the plane, so $\widetilde{Q}_u$ has three singularities only if one of the two nodes of $Q_u$ is at one of the $5$ points $p_i$. We can then either fix a conic through five points and let the line pass through one of them, or fix the line through two points and let the conic pass through four (both case (A)). The other possible reducible case is that $Q_u$ is a union of three lines. This makes exactly three nodes in the plane. Since any line can contain at most $2$ of the points $p_i$, the only possible partition of these points is $5 = 2+2+1$ which means that two lines are uniquely determined by the choice of points that they contain and the last one moves in the pencil of lines through the remaining point (case (B)). Regarding $X^{[3]}$ and $\left( X^{[3]} \right)^*$ we therefore conclude:

\begin{proposition}\label{prop:DelPezzo_3}
Let $X\subset\PP^4$ be a smooth Del Pezzo surface. Then $X^{[3]}\subset (\PP^4)^*$ is the union of $40$ lines $\PP^1\subset (\PP^4)^*$. Therefore, $\deg X^{[3]} = 40$.\\
The dual variety $\left( X^{[3]} \right)^*\subset \PP^4$ is the union of $40$ projective planes $\PP^2 \subset \PP^4$.
\end{proposition}

\subsubsection{The irreducible components for k=4}
The variety $X^{[4]}$ is a union of points $u$ such that $\widetilde{Q}_u$ has $4$ singularities. This implies that the plane cubic $Q_u$ must be reducible, as before ($k=3)$. In particular, there are two ways in which this can happen: the cubic has a smooth conic as an irreducible component or it is the union of three lines. The possible partitions are $4 = 2\cdot 2$ (case (A)), $4 = 1\cdot 2 + 2\cdot 1$ (case (B)), and $4 = 4\cdot 1$ (not realizable). 
\begin{figure}[ht]
    \centering
    \begin{tikzpicture}[scale=1]
\begin{axis}[
width=2in,
height=2in,
scale only axis,
xmin=-2,
xmax=2,
ymin=-2,
ymax=2,
ticks = none, 
ticks = none,
axis background/.style={fill=white},
axis line style={draw=none} 
]

\draw[thick,rotate around={-20:(0,0)}] (0,0) ellipse (1.5 and 1);
  
\draw[thick] (-1,-1.4) -- (1.46,1.2);

\addplot[only marks,mark=*,mark size=3pt,mygreen,
]  coordinates {
    (0.2,0.98) (1.44,-0.2) (0.9,0.63) (-0.5,-0.86) (-1.44,0.3577)
};

\node[black] (P) at (axis cs:0.95,0.4) {$p_i$};
\node[black] (P) at (axis cs:-0.5,-0.64) {$p_j$};

\end{axis}
\end{tikzpicture}
    \qquad
    \begin{tikzpicture}[scale=1]
\begin{axis}[
width=2in,
height=2in,
scale only axis,
xmin=-2,
xmax=2,
ymin=-2,
ymax=2,
ticks = none, 
ticks = none,
axis background/.style={fill=white},
axis line style={draw=none} 
]
  
\draw[thick] (0.4,1.5) -- (-0.7,-1.4);
\draw[thick] (1.9,-0.3) -- (-1.9,0.45);
\draw[thick] (-1.8,0.3) -- (1.3,0.7);

\addplot[only marks,mark=*,mark size=3pt,mygreen,
]  coordinates {
    (0.2,0.98) (1.44,-0.2) (0.9,0.63) (-0.5,-0.86) (-1.44,0.3577)
};

\node[black] (P) at (axis cs:-1.4,0.15) {$p_i$};

\end{axis}
\end{tikzpicture}
    \caption{Left: a conic through five points and a line through two points. Right: three lines through five points, two of them intersecting at $p_i$.}
    \label{fig:X4}
\end{figure}
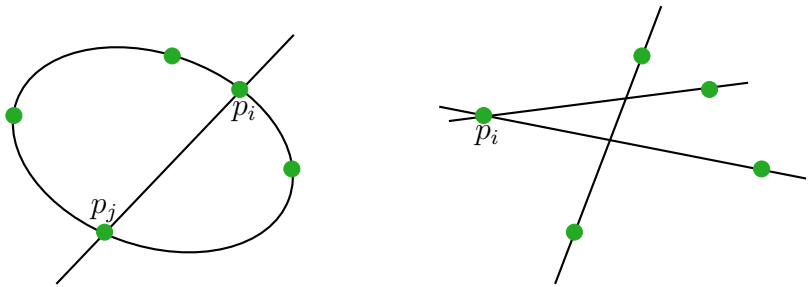

\paragraph{(A): conic + line.}
If $Q_u$ is the union of a conic and a line, the only way for obtaining four singularities on $C_u$ is that the $2$ intersection points of the two components are among the base points $p_1,\ldots,p_5$, as in Figure \ref{fig:X4}, left. So the conic must pass through all five points and is therefore unique. This leaves $\binom{5}{2}$ choices for the line which is uniquely determined by any two points $p_i$, $p_j$. These choices give therefore $10$ points of $X^{[4]}$.

\paragraph{(B): three lines.}
Consider now plane cubics that split into three lines. Since they determine only three intersection points in $\PP^2$, two of these lines must intersect in one of the base points $p_i$, as in Figure \ref{fig:X4}, right. To count them, we can first choose the point containing two lines and then choose the two points that determine the line that does not go through our first choice. Then, all three lines are determined so that this gives precisely $5\cdot \binom{4}{2} = 30$ points in $X^{[4]}$.

There are no other possible configurations of plane cubics $Q_u$ such that $\widetilde{Q}_u$ has four singularities. Indeed, the cubic must be reducible again, so that it is either a conic and a line, or three lines. The conic and the line intersect in two points, hence these are forced to be two of the $p_i$'s, in order for $\widetilde{Q}_u$ to have four singularities. Then the conic is forced to pass also through the remaining three points, so it is fixed. In the case of three lines, two lines are forced to go through two of the base points each. The remaining line must go through the fifth point; to obtain in the end $4$ singularities, this line must also intersect one of the other two lines in one of the $p_i$'s. These are therefore the only possibilities and they are the two cases discussed above. We summarize:

\begin{proposition}\label{prop:DelPezzo_4}
Let $X\subset \PP^4$ be a generic Del Pezzo surface. Then $X^{[4]}\subset (\PP^4)^*$ is the union of $40$ points. Therefore, $\left( X^{[4]}\right)^*\subset \PP^4$ is the union of $40$ hyperplanes.
\end{proposition} 

It turns out to be useful for the real picture in Section \ref{subsec:delPezzoReal} to relate these $40$ points to the quadrics of $X^{[2]}_{(D)}$.
\begin{proposition}\label{prop:delPezzox4vsx2}
Let $Y_i$ be the five irreducible components of $X_D^{[2]}$ in some order. (They are the dual varieties to the five singular quadrics in the pencil defining $X$.) The set $X^{[4]}$ is the union of the pairwise intersections of the $Y_i$:
\[
X^{[4]} = \bigcup_{i\neq j} Y_i\cap Y_j.
\]
Each intersection $Y_i\cap Y_j$ consists of four points and so the union is disjoint.
\end{proposition}

\begin{proof}
Let $u$ be a point of $X^{[4]}$ so that the hyperplane section $C_u = u^\perp\cap X$ is a curve with four singularities. The case analysis above shows that this hyperplane section is a union of four lines in $\PP^4$ (with either two exceptional divisors or one among them). Their intersection graph is a four cycle. The ideal of the curve $C_u = u^\perp \cap X$ in $u^\perp \cong \PP^3$ is generated by two quadrics, namely those in the pencil defining $X$ modulo the linear form given by $u$. This pencil contains two products of planes, where one of the planes is spanned by a pair of intersecting lines. The two planes of each pair intersect in a line spanned by two singular points of $C_u$, giving the missing two lines between the four points.
These two singular quadrics in the pencil associated to $u^\perp\cap X$ come from singular quadrics in the pencil associated to $X$, say $V_i$ and $V_j$. Since $V_i$ and $V_j$ contain a plane $\PP^2$, they must be singular. The line in $u^\perp$ in which the two planes intersect must pass through the cone point of the appropriate $V_k$ containing both planes. This shows that the line, in which the two planes intersect, is a line of the ruling of $V_k$. Now $u^\perp$ is the tangent hyperplane to $V_k$ along this line, which implies that $u$ is in $Y_k$. Since this argument works both for $V_i$ and $V_j$, it shows $u\in Y_i\cap Y_j$ as desired. Since the varieties $Y_1,\ldots,Y_5$ are surfaces of degree $2$, the intersection $Y_i\cap Y_j$ contains $4$ points. There are $10$ such pairs and the union $X^{[4]}$ has $40$ points, so that the union must be disjoint.
\end{proof}

\subsection[The real picture]{The $\R$eal picture}\label{subsec:delPezzoReal}
Up to now, we investigated the varieties $X^{[k]}$ and their duals for the del Pezzo surfaces $X\subset \PP^4$ from a point of view of complex projective geometry. Now we want to dive in the realm of convex and real algebraic geometry.
In what follows, we want to address the question: 
\begin{center}
    Which of the irreducible hypersurfaces in $\left( X^{[k]} \right)^*$ can contribute to $\partial_a \conv X$?
\end{center} 
This question is equivalent to ask: 
\begin{center}
    Are the points of $X^{[k]}\subset (\PP^4)^*$ supporting hyperplanes for $X$?
\end{center}
Recall that a hyperplane $u^\perp$ supports $\conv X$ if $X(\R)\cap u^\perp \neq \emptyset$ and $\lprod u, x \rprod$ has the same sign for all $x\in X(\R)$ (in affine coordinates on a chosen chart where $X(\R)$ is compact).

Since we are dealing with convex bodies, we want our smooth real algebraic del Pezzo surface $X$ to have a compact real locus on some affine chart of $\PP^4$. By the compactness of $\PP^4(\R)$, this is equivalent to the fact that the curve $H\cap X$ has no real points. The blow ups of the real projective plane in $a$ real points and $b$ pairs of complex conjugate points (with $a+2b = 5$) always contain a real line in $\PP^4$ and hence do not satisfy this requirement. However, there is a full classification of the real forms of Del Pezzo surfaces, see \cite{russo:delPezzo}. There are three relevant topological types:
the $4$ real forms in \cite[Corollary 3.2]{russo:delPezzo} that are not blow ups of the real projective plane do not contain real lines. The second to last case $(Q^{3,0} \times Q^{3,0})(0,4)$ in Russo's list has no real points at all, though. So the interesting cases for us are the three types $Q^{3,1}(0,4)$, $Q^{2,2}(0,4)$ and $\mathbb{D}_4$. Here, the notation $V(a,2b)$ for a real algebraic variety $V$ denotes the blow up of $V$ in $a$ real points and $b$ complex conjugate pairs of points. The notation $Q^{r,s}$ stands for the quadric in $\PP^{r+s-1}$ with signature $(r,s,0)$. The type $\mathbb{D}_4$ is a non-standard real structure (called de Jonqui\'eres involution) of the blow up of the real projective plane in five real points described in \cite[Example 2 for $n=3$]{russo:delPezzo}. The structure of the algebraic boundary of the convex hull of $X$ depends only on the topological type of our Del Pezzo surface $X$.

The complex algebraic discussion of the varieties $X^{[k]}$ in the previous subsections is relevant for all these real types:
a real algebraic Del Pezzo surface $X$ of degree $4$ is embedded in $\PP^4$ by its anticanonical linear system, so there is a complex projective automorphism $A\in {\mathrm PGL}_4(\C)$ that maps it to a (complex algebraic) surface that is the blow up of $\PP^2$ in $5$ points as a complex algebraic variety. Since $A$ is an automorphism of $\PP^4$, tangency of $X$ with hyperplanes $H\subset \PP^4$ is preserved and the above discussion of the irreducible components of subvarieties $X^{[k]}$ of $X^*$ for surfaces that are blow ups is still relevant here. For $k=2$, the only candidates for irreducible components of the algebraic boundary of $\conv X$ are the singular quadrics in the pencil of quadrics vanishing on $X$. 
Here we need to distinguish cases, depending on the signature of the real singular quadrics of the pencil. 

So what are the signatures in each pencil? Every topological type of a smooth real Del Pezzo surface forms a connected family, see \cite[Corollary 5.6]{Kollar:RealAlgebraicSurfaces}. Since a Del Pezzo surface of degree $4$ in $\PP^4$ is smooth if and only if there are precisely $5$ distinct singular quadrics in its defining pencil, the connectedness of the topological types implies that the number of real singular quadrics in the pencil of a real Del Pezzo surface $X\subset \PP^4$ and their signature is constant for every type. They can therefore be computed from an example, as we do below. Table \ref{table:Del_Pezzo_topo} shows the cases that are relevant for us.
\begin{table}[ht!]
    \renewcommand{\arraystretch}{1.7}
    \centering
    \begin{tabular}{c|c|c|c}
    \textbf{type} & \textbf{topology} & \textbf{real singular quadrics} & \textbf{signatures} \\
    \hline
    \hline
    $Q^{3,1}(0,4)$ & $S^2$ & $3$ & $(2,2,1), (3,1,1), (3,1,1)$ \\
    \hline
    $Q^{2,2}(0,4)$ & $S^1\times S^1$ & $5$ & $(2,2,1), (2,2,1), (2,2,1), (3,1,1), (3,1,1)$ \\
    \hline
    $\mathbb{D}_4$ & $S^2 \sqcup S^2$ & $5$ & $(2,2,1), (3,1,1), (3,1,1), (3,1,1), (3,1,1)$ \\
    \hline
    \end{tabular}
    \caption{Topological types of Del Pezzo surfaces with no real lines, and the signatures of the real singular quadrics of the associated pencil.}
    \label{table:Del_Pezzo_topo}
\end{table}

We now show that each topological type that does not contain a line has a compact real locus on a suitable affine chart of $\PP^4$. In order to find a hyperplane $H$ that contains no real points of $X$, we can reason as follows. Fix one of the real singular quadrics $Q$ of the pencil associated to $X$ having signature $(3,1,1)$ and consider the projection $p\colon \PP^4\dashrightarrow \PP^3$ from its singular point. The image of the quadric itself under this projection is a quadric with signature $(3,1,0)$, hence its real points form a sphere $S^2$. This projection, restricted to the Del Pezzo surface $X$, is a $2:1$ cover. Let $H'\subset \PP^3$ be a hyperplane that contains no real points of $p(Q)$. Then, the preimage $H = p^{-1}(H')\subset \PP^4$ is a hyperplane that contains no real points of $X$. The center of the projection does not lie on $X$ since $X$ is smooth and a complete intersection. In conclusion, the Del Pezzo surface $X$ has a compact real locus in the affine chart $\PP^4\setminus H$.

For dimensional reasons, we already know that $\left( X^{[3]}\right)^*$ is not part of the algebraic boundary of $\conv X$.
On the other hand, $\left( X^{[4]}\right)^*$ is a union of $40$ hyperplanes, but none of them is a supporting hyperplane for $X(\R)$. Indeed, Proposition~\ref{prop:delPezzox4vsx2} implies that the real hyperplanes $u\in X^{[4]}$ do not intersect the convex hull of $X$ at all or they support one of its edges. These edges are in the ruling of a real singular quadric in the pencil and therefore already captured by the varieties in $X^{[2]}_{(D)}$. Alternatively, the dimension of the supported face is $1$ so that its Zariski closure is not equal to $u^\perp$, so it cannot be an irreducible component of the algebraic boundary.

Let us focus now on $\left( X^{[2]}\right)^*$. 
The convexity of $\conv X$ imposes curvature conditions on the hypersurfaces in its algebraic boundary. A quadric with signature $(2,2,1)$ is a cone over an hyperboloid in $\PP^3$ which has negative Gaussian curvature at every point, and therefore a positive and a negative principal curvatures. However, the principal curvatures at a point in the boundary of a convex body must be, when defined, nonnegative \cite[Section 2.5]{Schneider:BrunnMinkowskiTheory}. Hence, these quadrics cannot bound a convex set.
Therefore, in the $Q^{3,1}(0,4)$ and $Q^{2,2}(0,4)$ cases there can be at most two of the varieties in $\left( X^{[2]} \right)^*$, namely the quadrics with signature $(3,1,1)$, that contribute to the algebraic boundary of the convex hull. 
In the case when $X$ is of type $\mathbb{D}_4$, there are at most $4$ quadratic cones that can contribute to the algebraic boundary of the convex hull. 
For any of these four quadrics, there exists an affine chart in which it contributes to the algebraic boundary. To show this, fix one of these real singular quadrics, say $V$. Consider the projection $p$ to $\PP^3$ from the singular point of $V$. As we noticed above, this gives a $2:1$ cover when restricted to $X$. The image $p(X)$ is a quadric with signature $(3,1,0)$, namely $S^2$. The preimage of each tangent hyperplane to a point on the sphere is a hyperplane in $\PP^4$. By chosing an affine chart for instance with $p$ at infinity, this hyperplane is a supporting hyperplane for $X$, because it was a supporting hyperplane in the projection. The same reasoning holds for all the singular quadrics in the pencil with signature $(3,1,1)$.

In fact, not all four singular quadrics can show up in the boundary of $\conv X$ in the same affine chart: they come in two pairs, and only one pair can be part of the algebraic boundary  of $\conv X$ in a given affine chart. We make this rigorous in the following proposition.
\begin{proposition}\label{prop:pairs_quadrics_D4}
    Let $X\subset\PP^4$ be a smooth real Del Pezzo surface of degree $4$ and type $\mathbb{D}_4$. Let $V_1,V_2,V_3,V_4$ be the four singular quadrics of the pencil having signature $(3,1,1)$. Then, there are exactly two pairs of quadrics such that in a given affine chart (where $X(\R)$ is compact) only one pair can contribute to $\partial_a (\conv X)$.
\end{proposition}
\begin{proof}
    The signatures of the matrices in the pencil associated to $X$ are shown in Figure \ref{fig:pencil}. 
    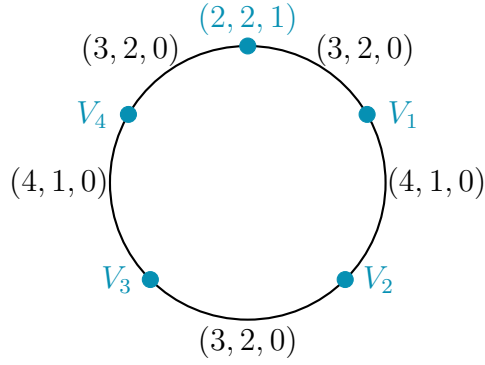
\begin{figure}[ht]
        \centering
        \begin{tikzpicture}[scale=1]
\begin{axis}[
width=2.57in,
height=2in,
scale only axis,
xmin=-1.8,
xmax=1.8,
ymin=-1.4,
ymax=1.4,
ticks = none, 
ticks = none,
axis background/.style={fill=white},
axis line style={draw=none} 
]

\draw[thick] (0,0) circle (1);

\addplot[only marks,mark=*,mark size=3pt,myteal,
]  coordinates {
    (0,1) (0.866,0.5) (0.707,-0.707) (-0.707,-0.707) (-0.866,0.5)
};

\node[myteal] (P) at (axis cs:0,1.2) {$(2,2,1)$};
\node[myteal] (P) at (axis cs:1.13,0.5) {$V_1$};
\node[myteal] (P) at (axis cs:0.95,-0.707) {$V_2$};
\node[myteal] (P) at (axis cs:-0.95,-0.707) {$V_3$};
\node[myteal] (P) at (axis cs:-1.13,0.5) {$V_4$};

\node[black] (P) at (axis cs:0.85,0.97) {$(3,2,0)$};
\node[black] (P) at (axis cs:1.37,0) {$(4,1,0)$};
\node[black] (P) at (axis cs:0,-1.16) {$(3,2,0)$};
\node[black] (P) at (axis cs:-1.37,0) {$(4,1,0)$};
\node[black] (P) at (axis cs:-0.85,0.97) {$(3,2,0)$};

\end{axis}
\end{tikzpicture}
        \caption{The pencil of quadrics in the $\mathbb{D}_4$ type of Del Pezzo surface, with the respective signatures.}
        \label{fig:pencil}
    \end{figure}
Let us consider an affine chart on this $\PP^1$ by assuming that the singular quadric with signature $(2,2,1)$ is the point at infinity.
    
Fix now the quadric $V_1$. Up to a change of coordinates in $\PP^4$, it has equation
\begin{equation}
    x_0^2+x_1^2+x_2^2-x_3^2=0.
\end{equation}
Denote by $p_i$ the singular point of $V_i$. With this choice of coordinates $p_1 = [0,0,0,0,1]$.
Consider the quadric $Q_i = V_1 + \varepsilon V_i$, for $i=2,3,4$ and $\varepsilon >0$. For $\varepsilon$ small enough, because of our choice of affine chart in the pencil, the quadric $Q_i$ has signature $(4,1,0)$. The determinant of the matrix associated to $Q_i$ will then be
\begin{equation}
    -\varepsilon c_i + \varepsilon^2 (\ldots )
\end{equation}
where $c_i$ is the coefficient of $x_4^2$ in $V_i$.
Therefore, the signature of $Q_i$ is determined by the sign of $c_i$: in this case for all $i$ we have $c_i>0$.
Notice that $V_i(p_1)=c_i$ for $i=2,3,4$, hence $V_i(p_1)>0$. Here we are using the notation $V_i$ for both the quadrics and its defining polynomial.
We can repeat the same argument for $V_2,V_3,V_4$ and we get the sign conditions shown in Table \ref{table:pairs_of_quadrics}.
\begin{table}[ht!]
    \renewcommand{\arraystretch}{1.7}
    \centering
    \begin{tabular}{c|c|c|c}
    \textbf{signs at $p_1$} & \textbf{signs at $p_2$} & \textbf{signs at $p_3$} & \textbf{signs at $p_4$} \\
    \hline
    \hline
    $V_2(p_1) > 0$ & $V_1(p_2) > 0$ & $V_1(p_3) < 0$ & $V_1(p_4) > 0$ \\
    \hline
    $V_3(p_1) > 0$ & $V_3(p_2) < 0$ & $V_2(p_3) < 0$ & $V_2(p_4) > 0$ \\
    \hline
    $V_4(p_1) > 0$ & $V_4(p_2) < 0$ & $V_4(p_3) > 0$ & $V_3(p_4) > 0$ \\
    \hline
    \end{tabular}
    \caption{Signs of the singular quadrics with signature $(3,1,1)$ evaluated at the singular points of the other quadrics.}
    \label{table:pairs_of_quadrics}
\end{table}

Let us consider now an affine chart in which $V_2$ contributes to $\partial_a \conv X$.
Since $V_2$ has signature $(3,1,1)$, the convex hull of the Del Pezzo surface $X$ in this chart must be contained in one of the regions in which $V_2<0$. Consider the quadric $V_3$ in this affine chart. The cone point $p_3$ belongs to the locus where $V_2$ is negative, therefore a real line of $V_3$ through $p_3$ will intersect $X$ in two real points. The segment connecting these two points will contain $p_3$ in its interior.
Since the discriminant of $V_3$ with respect to the projection from $p_2$ is definite (and hence has no real points), the two extrema of this segment belong to two distinct connected components.
Assume by contradiction that this segment belongs to the boundary of $\conv X$. Then the hyperplane $H$ exposing that segment is the tangent hyperplane to $V_3$ at the points of the segment (they all have the same tangent hyperplane, since they lie on a line of a quadratic cone through its vertex). Then $H$ has a different sign on the two real connected components of $X$. Indeed, consider a line through $p_3$ such that $V_3(x)\leq 0$ for every $x$ on that line. Then $H$ changes sign on that line in $p_3$. So $H$ cannot be a supporting hyperplane.

This shows that if in the chosen affine chart $V_2$ is part of the algebraic boundary of $\conv X$, then $V_3$ cannot be part of it. Repeating the same argument for all quadrics we find that in the same affine chart the following pairs of quadrics can bound simultaneously $\conv X$:
\begin{equation}
    V_1, V_2 \qquad \hbox{or} \qquad V_3, V_4 \qquad \hbox{or} \qquad V_1, V_4.
\end{equation}

In the next step, we exclude the pair $V_1, V_4$. 
Table~\ref{table:pairs_of_quadrics} shows that $V_3(p_1)>0$ and $V_1(p_3)<0$. Geometrically, each $V_i$ is a cone over a quadric with signature $(3,1,0)$ in $\PP^3$, namely, a sphere. By the choice of the signature, in every affine chart $\R^4$ of $\PP^4$ the region that contains the points inside the sphere is convex and it is given by a connected component of $V_i\leq 0$. Therefore, $p_3$ is inside the sphere of $V_1$ and $p_1$ is outside the sphere of $V_3$. This implies that the two real intersection points of a line on $V_3$ through $p_3$ lie on the two connected components of $X(\R)$. Since $p_1$ is outside of the sphere of $V_3$ and $p_3$ inside the sphere of $V_1$, two real intersection points with a line on $V_1$ through $p_1$ lie on the same connected component of $X(\R)$. In particular, the projection from $p_1$ of the real points of $X$ has two connected components.
Of course, the same argument works symmetrically for $V_4$ instead of $V_1$.
Regarding $V_2$ and $V_3$, the argument above with the hyperplane $H$ shows that the projection of $X(\R)$ from $p_2$ or $p_3$ is surjective.

Now fix an affine chart where $X(\R)$ is compact and such that $V_1$ is an irreducible component of $\partial_a \conv X$. Assume by contradiction that the other boundary component of $\partial_a\conv X$ is $V_4$. 
The quadric $V_2$ is a cone over a sphere, with singular point $p_2\in \{x \in\PP^4(\R)\colon V_1(x)<0\}$. Since the set $\{x \in\PP^4(\R)\colon V_2(x)<0\}$ in any affine chart $\R^4$ containing $p_2$ is the union of two convex cones $C_1, C_2$, with $C_1\cap C_2 = p_2$, there exists an affine linear function $\ell$ on $\R^4$ that separates the cones, i.e., $\ell(p_2)=0$, $\ell(C_1\setminus p_2)>0$ and $\ell(C_2\setminus p_2)<0$. Since $p_1$ belongs to the region where $V_2$ is positive, we can also assume that $\ell(p_1) = 0$. Let $S_1$ and $S_2$ be the two connected components of the real locus $X(\R)$ on our affine chart $\R^4$. Since $X(\R)$ is compact on that chart, the reasoning above shows that $S_1 \subset C_1$ and $S_2 \subset C_2$, and therefore $\ell$ separates the two connected components of $X(\R)$.

Let $p = \lambda y_1 + (1-\lambda) y_2$ for some $\lambda\in (0,1)$, $y_i\in S_i$, such that $\ell(p)=0$. Consider then the line $L$ spanned by $p$ and $p_1$, which satisfies $\ell(L)=0$. Since $p_1$ does not belong to $\conv X$, there exists a real point $q\in L$ such that $q\in \partial \conv X$. Since $\partial_a \conv X\subset (X^{[2]})^*$, we can write $q$ as a convex combination of real points $x_1,x_2 \in X(\R)$ where these two points belong to two different connected components of $X(\R)$.
Since $\partial_a \conv X = V_1 \cup V_4$, then $q\in V_i$ for $i=1$ or $4$. Therefore, $x_1, x_2, q \in V_i$ and since this is a quadric then the line spanned by the three points is in the ruling of $V_i$. By the previous part of the proof, for $i=1,4$ if a line of the ruling of $V_i$ meets $X$ in two real points, then the latter are forced to lie on the same connected component of $X$. This gives a contradiction, since $x_1, x_2$ must lie on two different connected components.
In conclusion, we have that
\begin{equation}
    \partial_a \conv X = V_1 \cup V_2 \qquad \hbox{or} \qquad \partial_a \conv X = V_3 \cup V_4.
\end{equation}
\end{proof}
\begin{remark}
    We point out that the pairs of quadrics in Proposition \ref{prop:pairs_quadrics_D4} are formed by consecutive quadrics in the pencil, such that the segment connecting them (the one that does not have other singular quadrics) is given by quadrics with signature $(4,1,0)$. In both pairs we have one quadric whose ruling connects points of the same connected component of $X$ (namely $V_1$ and $V_4$), and one quadric whose ruling connects points of different connected components of $X$ (namely $V_2$ and $V_3$).
\end{remark}

In conclusion, the algebraic boundary of a Del Pezzo surface consists of two singular quadrics of the associated pencil, with signature $(3,1,1)$. 
We summarize our results for a Del Pezzo surface in the following theorem.

\begin{theorem}\label{thm:del_pezzo}
Let $X\subset\PP^4$ be a smooth real Del Pezzo surface of degree $4$. Then there is an affine chart $\:\PP^4\setminus H$ of $\:\PP^4$ such that $X(\R)\cap (\PP^4\setminus H)$ is compact if and only if $X$ does not contain any real line.
There are three such topological types with $X(\R)$ homeomorphic to $S^2$, $S^1\times S^1$ or the disjoint union $S^2\sqcup S^2$ of two copies of $S^2$, namely $Q^{3,1}(0,4)$, $Q^{2,2}(0,4)$, and $\mathbb{D}_4$. Then 
\begin{equation}
   \partial_a (\conv X) = V_1 \cup V_2, \qquad \deg \partial_a (\conv X) = 4.
\end{equation}
Here $V_1, V_2$ are two singular quadrics of the pencil having signature $(3,1,1)$. For the types $Q^{3,1}(0,4)$, $Q^{2,2}(0,4)$ there is a unique choice; for the type $\mathbb{D}_4$ they are one of the two possible pairs, according to Proposition \ref{prop:pairs_quadrics_D4}.
\end{theorem}
We want to point out the big improvement that we obtained with this case-by-case analysis: Table \ref{table:degreesXk} gives a variety of degree $12+26+40+40$ in $\PP^4$ to be dualized in order to find the algebraic boundary of $X$. In the end, taking into account real and convex aspects, only $4$ of the $118$ degrees and $2$ of the $113$ irreducible components are relevant.

\begin{example}\label{ex:DelPezzo}
Consider the two quadrics $V_0,V_{\infty}$ defined respectively by the polynomials
\begin{align}
    f_0 &= x_0^2-x_1^2-x_2^2-x_3^2-x_4^2,\\
    f_{\infty} &= 2 x_2^2-2 x_1 x_3+2 x_0 x_4.
\end{align}
Let $X\subset \PP^4$ be the smooth Del Pezzo surface with ideal $I = \langle f_0, f_{\infty} \rangle$.
\begin{figure}[h]
    \centering
    \includegraphics[width=0.36\textwidth]{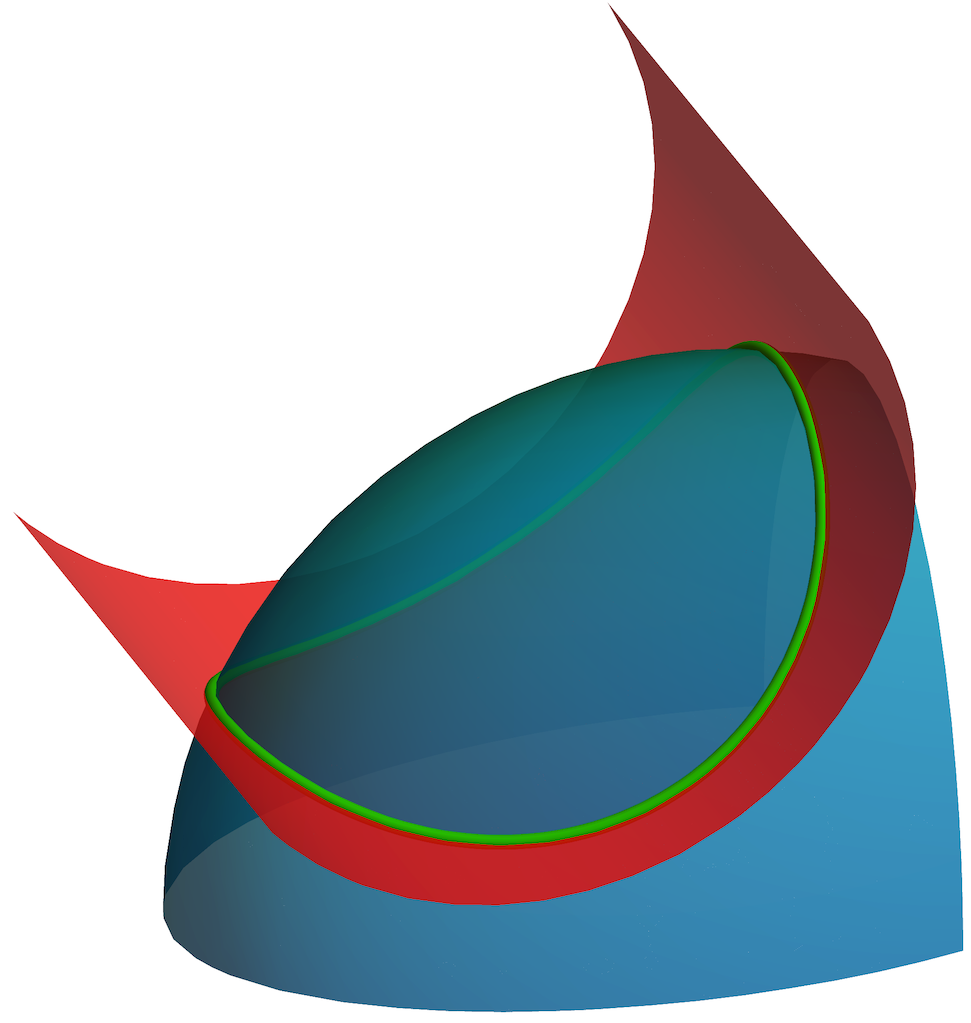}
    \caption{A hyperplane section of the Del Pezzo surface in Example \ref{ex:DelPezzo}: here the green curve. The blue and the red quadric are \eqref{eq:q1} and \eqref{eq:q2} respectively.}
    \label{fig:DelPezzo}
\end{figure}
The associated pencil of quadrics $\mathscr{L}$ is defined by the pencil of matrices
\begin{equation}
    \begin{bmatrix}
    -\lambda & 0 & 0 & 0 & \mu \\
    0 & \lambda & 0 & -\mu & 0 \\
    0 & 0 & \lambda+2\mu & 0 & 0 \\
    0 & -\mu & 0 & \lambda & 0 \\
    \mu & 0 & 0 & 0 & \lambda
    \end{bmatrix}
\end{equation}
for $[\lambda,\mu ] \in \PP^1$.
This Del Pezzo surface is of type $Q^{3,1}(0,4)$, hence it is topologically $S^2$ and among the five singular quadrics of the pencil, three are real. The signature of their matrices is either $(3,1,1)$ or $(2,2,1)$. These quadrics are:
\begin{align}
    (3,1,1) &: && x_0^2-x_1^2-3 x_2^2+2 x_1 x_3-x_3^2-2 x_0 x_4-x_4^2 = 0, \label{eq:q1} \\
    & && 2 x_0^2-2 x_1^2-2 x_1 x_3-2 x_3^2+2 x_0 x_4-2 x_4^2 = 0, \label{eq:q2} \\
    (2,2,1) &: && x_0^2-x_1^2+x_2^2-2 x_1 x_3-x_3^2+2 x_0 x_4-x_4^2 = 0,
\end{align}
and can be obtained for $[\lambda,\mu] = [1,1], [1,-1], [2,-1]$ respectively. Figure \ref{fig:DelPezzo} shows the hyperplane section $2 x_1 - 1 = 0$ of the Del Pezzo surface and the two $(3,1,1)$ quadrics, in the affine chart $\{x_0 \neq 0\}$. Notice that the hyperplane section of the convex hull of $X$ is not the convex hull of the hyperplane section of $X$.
\end{example}

\begin{example}\label{ex:DelPezzoS1S1}
Consider the two quadrics $V_0,V_{\infty}$ defined respectively by the polynomials:
\begin{align}
    f_0 &= 2 x_0^2 - 3 x_1^2 - x_2^2 + x_4^2,\\
    f_{\infty} &= 3 x_0^2 - 2 x_1^2 - x_3^2 - x_4^2.
\end{align}
Let $X\subset \PP^4$ be the smooth Del Pezzo surface with ideal $I = \langle f_0, f_{\infty} \rangle$.
\begin{figure}[!h]
    \centering
    \includegraphics[width=0.35\textwidth]{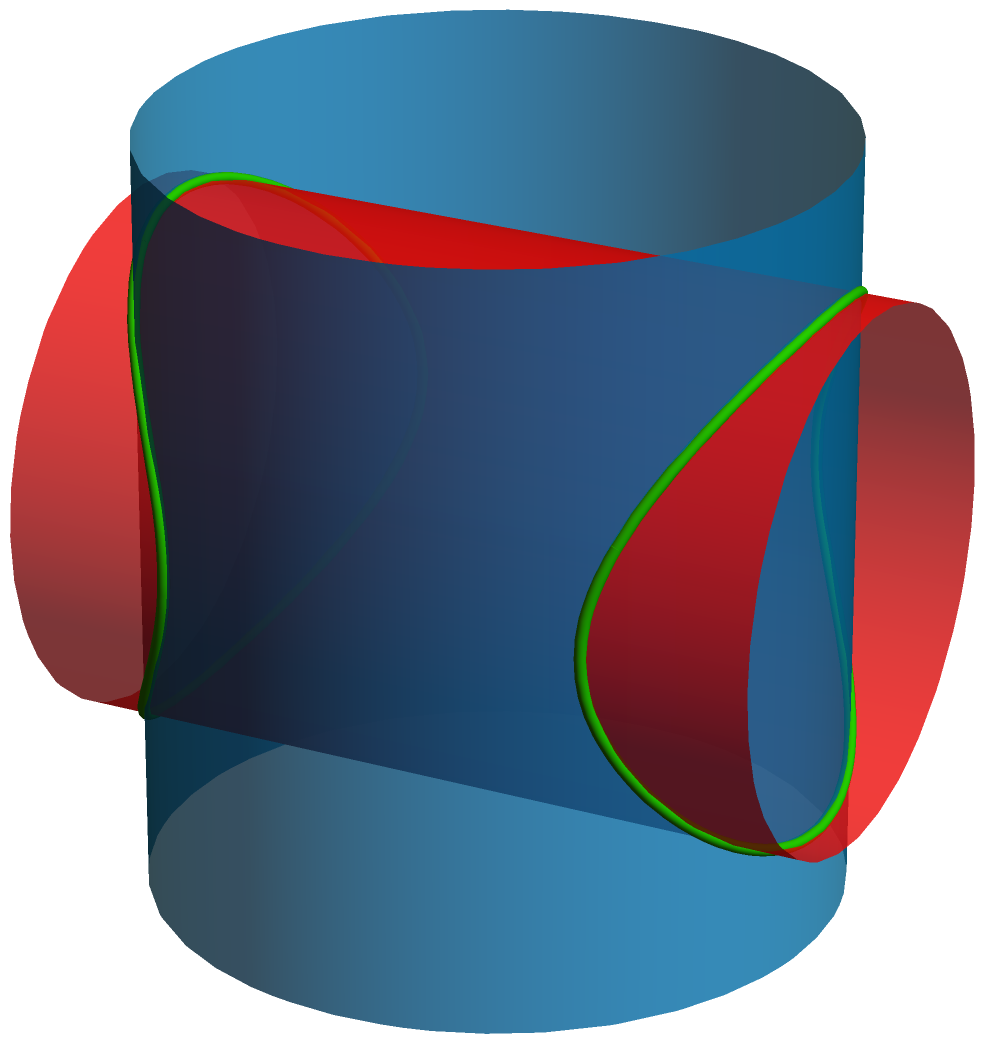}
    \caption{The hyperplane section $\{x_2=1\}$ of the Del Pezzo surface in Example \ref{ex:DelPezzoS1S1}, in the affine chart $\{x_0\neq 0\}$: here the green curve. The red and the blue cylinders are the zero loci of $f_{\infty}$ and the other $(3,1,1)$ quadric, respectively.}
    \label{fig:DelPezzoS1S1}
\end{figure}
The associated pencil of quadrics $\mathscr{L}$ is defined by the pencil of matrices
\begin{equation}
    \begin{bmatrix}
    2 \lambda + 3 \mu & 0 & 0 & 0 & 0 \\
    0 & - 3 \lambda - 2 \mu & 0 & 0 & 0 \\
    0 & 0 & - \lambda - \mu & 0 & 0 \\
    0 & 0 & 0 & -\mu & 0 \\
    0 & 0 & 0 & 0 & \lambda - \mu
    \end{bmatrix}
\end{equation}
for $[\lambda,\mu ] \in \PP^1$.
This Del Pezzo surface is of type $Q^{2,2}(0,4)$, hence it is topologically $S^1\times S^1$ and all the five singular quadrics of the pencil are real. The are three matrices with signature $(2,2,1)$ and two with signature $(3,1,1)$. These quadrics are:
\begin{align}
    (3,1,1) &: && 5 x_0^2 - 5 x_1^2 - x_2^2 - x_3^2 = 0, \qquad\qquad\\
    & && 3 x_0^2 - 2 x_1^2 - x_3^2 - x_4^2 = 0, \qquad\qquad\\
    (2,2,1) &: && 2 x_0^2 - 3 x_1^2 - x_2^2 + x_4^2 = 0, \qquad\qquad\\
    & && 5 x_1^2 + 3 x_2^2 - 2 x_3^2 - 5 x_4^2 = 0, \qquad\qquad\\
    & && 5 x_0^2 + 2 x_2^2 - 3 x_3^2 - 5 x_4^2 = 0, \qquad\qquad
\end{align}
and can be obtained for $[\lambda,\mu] = [1,1], [0,1] , [1,0], [-3,2], [-2,3]$ respectively. The real part of $X$ is compact in the affine chart $\{x_0=1\}$. Figure \ref{fig:DelPezzoS1S1} shows $X$ and the two $(3,1,1)$ quadrics in the hyperplane section $\{x_2=1\}$ of the affine chart $\{x_0\neq 0\}$.
\end{example}

\begin{example}\label{ex:DelPezzoS2S2}
Consider the two quadrics $V_0,V_{\infty}$ defined respectively by the polynomials:
\begin{align}
    f_0 &= 2 x_0^2 - x_1^2 - x_3^2 - x_4^2,\\
    f_{\infty} &= x_0^2 - 2 x_1^2 - x_2^2 - x_4^2.
\end{align}
Let $X\subset \PP^4$ be the smooth Del Pezzo surface with ideal $I = \langle f_0, f_{\infty} \rangle$.
\begin{figure}[h!]
    \centering
    \includegraphics[width=0.4\textwidth]{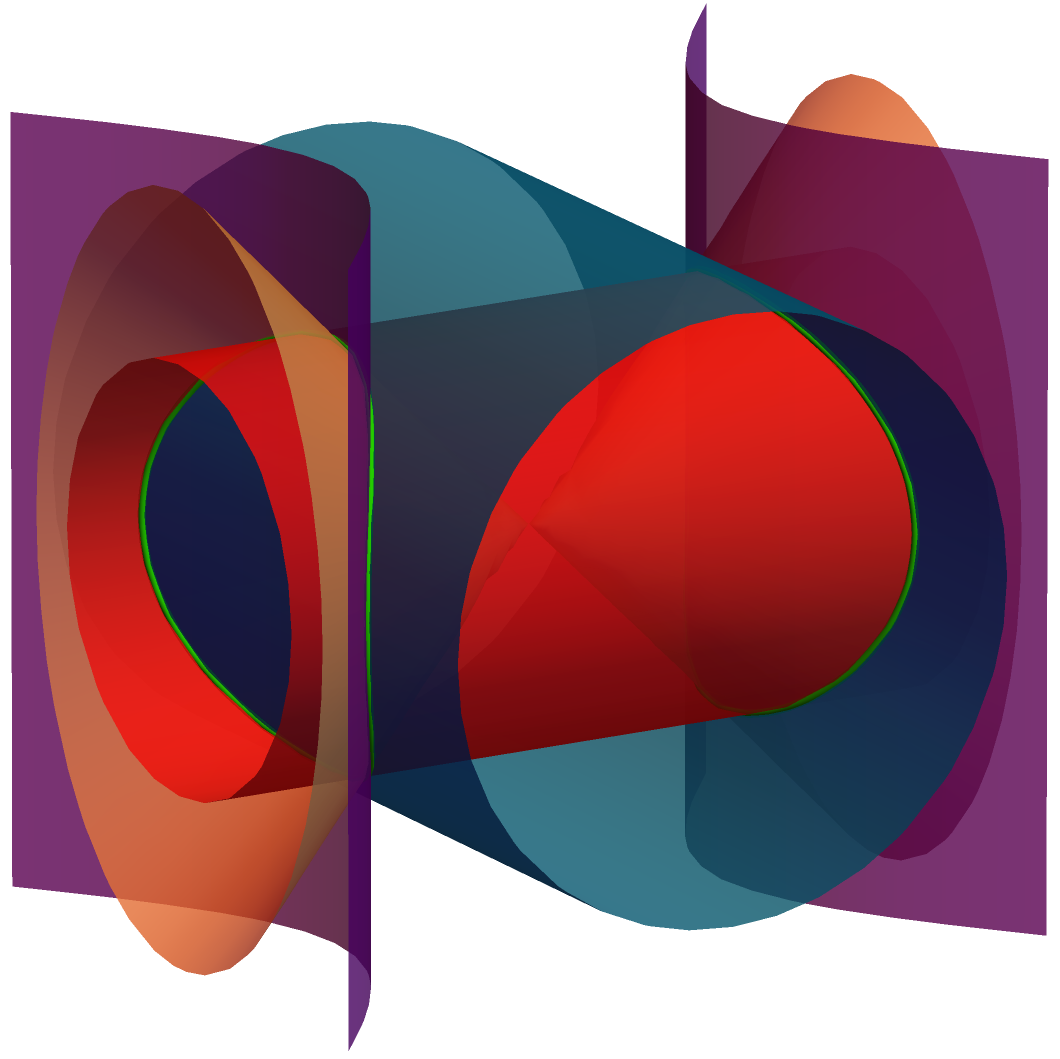}
    \includegraphics[width=0.45\textwidth]{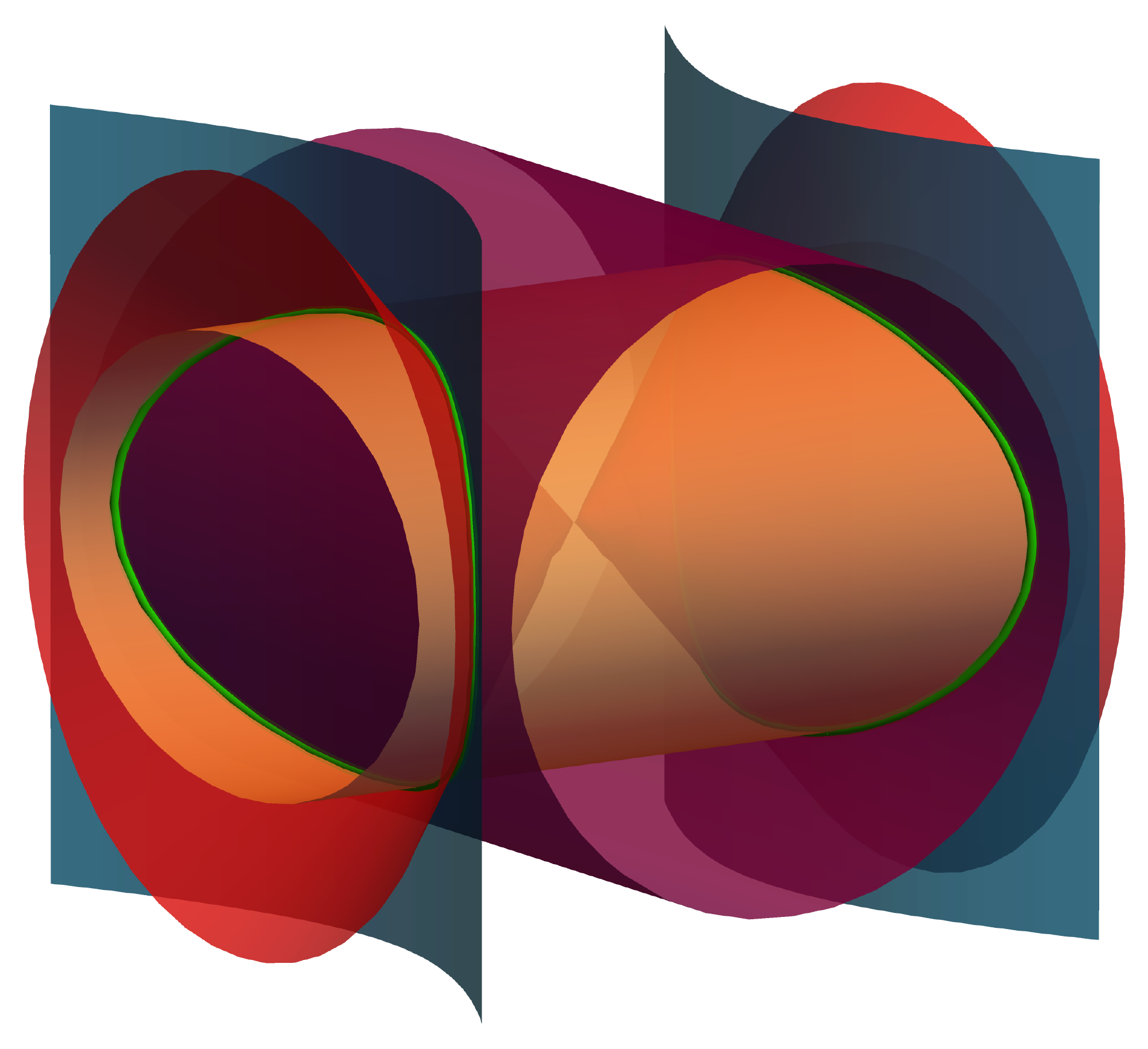}
    \caption{The hyperplane section $\{x_1=1\}$ of the Del Pezzo surface in Example \ref{ex:DelPezzoS2S2}, in the affine charts $\{x_0\neq 0\}$ (left) and $\{x_3\neq 0\}$ (right). The section of the Del Pezzo surface is here the green curve. The blue and red as well as the orange and purple $(3,1,1)$ quadrics are the two pairs of possible algebraic boundaries of $X$, according to Proposition \ref{prop:pairs_quadrics_D4}.}
    \label{fig:DelPezzoS2S2}
\end{figure}
The associated pencil of quadrics $\mathscr{L}$ is defined by the pencil of matrices
\begin{equation}
    \begin{bmatrix}
    2 \lambda + \mu & 0 & 0 & 0 & 0 \\
    0 & - \lambda - 2 \mu & 0 & 0 & 0 \\
    0 & 0 & - \mu & 0 & 0 \\
    0 & 0 & 0 & - \lambda & 0 \\
    0 & 0 & 0 & 0 & - \lambda - \mu
    \end{bmatrix}
\end{equation}
for $[\lambda,\mu ] \in \PP^1$.
This Del Pezzo surface is of type $\mathbb{D}_4$, hence it is topologically $S^2\sqcup S^2$ and all the five singular quadrics of the pencil are real. The signature of one of their matrices is $(2,2,1)$, and it is $(3,1,1)$ for the other four. These quadrics are:
\begin{align}
    (3,1,1) &: && \textcolor{myteal!80!blue}{V_1}\!: \; 2 x_0^2 - x_1^2 - x_3^2 - x_4^2 = 0, \qquad\qquad\\
    & && \textcolor{myred!70!red}{V_2}\!: \; x_0^2 - 2 x_1^2 - x_2^2 - x_4^2 = 0, \qquad\qquad\\
    & && \textcolor{myorange!95!red}{V_3}\!: \; 3 x_1^2 + 2 x_2^2 - x_3^2 + x_4^2 = 0, \qquad\qquad\\
    & && \textcolor{mypurple}{V_4}\!: \; x_0^2 + x_1^2 + x_2^2 - x_3^2 = 0, \qquad\qquad\\
    (2,2,1) &: && \qquad 3 x_0^2 + x_2^2 - 2 x_3^2 - x_4^2 = 0, \qquad\qquad
\end{align}
and can be obtained for $[\lambda,\mu] = [0,1] , [1,0], [1,-2], [1,-1], [2,-1]$ respectively. The teal labels refer to Figure \ref{fig:pencil}. The real part of $X$ is compact in the affine chart $\{x_0=1\}$. Figure \ref{fig:DelPezzoS2S2} shows $X$ and the four $(3,1,1)$ quadrics in the two possible configurations.
\end{example}

\section{Bordiga}\label{sec:bordiga}
The smooth surface $X$ considered in the current section is named after Giovanni Bordiga. This surface is relevant for instance in Computer Vision. In \cite{BNT:BordigaComputerVision}, the authors study the reconstruction of an image obtained from three cameras; this can be modelled as three projections from $\PP^4$ to $\PP^2$. In this framework, the Bordiga surface is the critical locus, i.e., the set of points that do not admit a unique reconstruction from the planar images.

Every Bordiga surface can be realized as a complex algebraic surface as the blow up $\mathrm{Bl}_{p_1,\ldots,p_{10}} \PP^2$ of ten generic points in the complex projective plane, embedded in $\PP^4$ via the divisor $D = 4\pi^*(L) - \sum_{i=1}^{10} E_i$. 
Equivalently, it can be defined as the zero locus of the $3\times 3$ minors of a $3\times 4$ matrix with homogeneous linear entries in five variables \cite[Theorem 9.3.5]{Dolgachev:ClassicalAlgebraicGeometry}.
For more details on the geometry of this surface we refer to \cite{Bordiga:SuperficieSestoOrdine, Dolgachev:ClassicalAlgebraicGeometry}. For the real picture, we only consider real algebraic surfaces that arise as the blow up of the real projective plane in a real set of $10$ points. 

We assume that the ten points $p_1,\ldots,p_{10}$ are in general position in the following sense: no three points are collinear, no six points on a conic, and no ten points on a cubic. Furthermore, the unique cubic through any nine out of the ten points is non-singular and there is no quartic through the ten points that has a triple point at one of them. Finally, there is a unique quartic that has nodes at $p_i$ and $p_j$ and passes through the other eight points for any choice of $i\neq j$. Each of these assumptions is a statement about the rank of a linear system of equations in the coefficients of a curve and/or a property of the unique solution of such a system. This can therefore be checked easily for any input of 10 points in $\PP^2$. To make sure that this holds generically, it is enough to produce one set of ten points satisfying all conditions above.

\subsection[The complex picture]{The $\C$omplex picture}
This situation is similar to the Del Pezzo case: we fix a basis $\{q_0,\ldots,q_4\}\subset \C[x_0,x_1,x_2]_{4}$ for the space of plane quartics through $p_1,\ldots,p_{10}$. With this choice of basis, the rational map $\varphi : \PP^2 \dashrightarrow \PP^4$ induced by $D$ maps $x\in \PP^2$ to $[q_0(x),\ldots,q_4(x)]$.
Its base locus coincides with the ten points $p_1,\ldots,p_{10}$. The Bordiga surface $X$ is the closure of the image of $\varphi$. Then, $\varphi$ is birational to $X$ and we have the following diagram:
\begin{equation}
\begin{tikzcd}
\textrm{Bl}_{p_1, \ldots ,p_{10}} \PP^2 \arrow[dr, "\widetilde{\varphi}"] \arrow[d, "\pi"] & \\
\PP^2 \arrow[r, dashed, "\varphi"] & X \subset \PP^4
\end{tikzcd}
\end{equation}
Here $\widetilde{\varphi}$ is an isomorphism between the blow up and $X$. Notice that the image of an exceptional curve $E_i \subset \textrm{Bl}_{p_1, \ldots ,p_{10}} \PP^2$ is a line in $\PP^4$. As for the Del Pezzo case in Section \ref{sec:del_pezzo}, a hyperplane $u^\perp$ identifies two curves: the hyperplane section $C_u = u^\perp\cap X \subset \PP^4$ and the plane quartic $Q_u = \varphi^*(C_u)$; we can compare their singularities via \eqref{eq:multiplicities_singularities}, and denote
\begin{equation}
    \widetilde{Q}_u = \widehat{Q}_u + \sum_{i=1}^{10} (\textrm{mult}_{p_i} Q_u - 1) E_i,
\end{equation}
where $\widehat{Q}_u$ is the strict transform of $Q_u$ and $\widetilde{Q}_u$ is the pullback via $\widetilde{\varphi}$ of the class of $C_u$.
Our goal is to analyze the irreducible components of the varieties $X^{[k]}$ by looking at the plane curves $Q_u$ that correspond to points $u\in X^{[k]}$. Since the aim is to describe $\partial_a (\conv X)$, we only care about the components whose dual varieties are hypersurfaces. The results in the cases $k=2,3,4$ are summarized in Propositions \ref{prop:Bordiga_2}, \ref{prop:Bordiga_3}, \ref{prop:Bordiga_4}. 

A node on $Q_u$ at a base point corresponds to two singularities of $C_u$, whereas a node outside the base locus gives one singularity of $C_u$, so our following case analysis for $X^{[k]}$ is again based on partitions of $k$ into parts that are either $1$ or $2$. Since the curve $Q_u$, corresponding to a hyperplane section of the Bordiga surface, is a plane quartic of arithmetic genus $3$, the case analysis is more involved compared to the case of the Del Pezzo surfaces in Section~\ref{sec:del_pezzo}, for which $Q_u$ is a plane cubic of arithmetic genus $1$.

\subsubsection{The irreducible components for k=2}

We first describe $X^{[2]}$ which is, as expected, a surface in $(\PP^4)^*$. Its points come from plane quartics $Q_u$, such that $\widetilde{Q}_u$ has two nodes. There are two families of plane quartics for which this happens, corresponding to the two partitions $2 = 1\cdot 2$ (case (A)) and $2 = 2\cdot 1$ (case (B)).

\paragraph{(A): a node at $p_i$.} Assume that the plane quartic $Q_u$ has one node at one of the base points $p_i$. This is a linear condition on its coefficients, hence the set of corresponding $u\in (\PP^4)^*$ is the union of $10$ copies of projective planes $\PP^2\subset (\PP^4)^*$, one for each choice of the node at $p_1,\ldots,p_{10}$.

\paragraph{(B): two nodes in $\PP^2$.} Assume that $Q_u$ has two nodes somewhere in $\PP^2$. Using Kazarian's formulae for $D = 4 L$, $K_{\PP^2} = -3 L$, i.e., $\mathsf{d} = 16, \mathsf{k} = -12, \mathsf{s} = 9, \mathsf{x} = 3$, we find that this surface has degree $225$. It remains an open problem to compute $\deg (X^{[2]}_{(B)})^*$ (that is to say the degree of the dual varieties for each irreducible component of $X^{[2]}$ for which $Q_u$ is a curve of type (B), for a generic point $u$ on this component). 

These are all the possible cases in which a curve $\widetilde{Q}_u$ can have two singularities, for a plane quartic $Q_u$. If $Q_u$ is irreducible, then it can have either a node at one of the fixed points $p_i$, or two nodes in the plane. As soon as the curve becomes reducible, it has at least three singularities, hence it is not relevant for $X^{[2]}$.
On $X^{[2]}$ and $\left( X^{[2]} \right)^*$ we therefore have:

\begin{proposition}\label{prop:Bordiga_2}
Let $X\subset \PP^4$ be a generic Bordiga surface. Then $X^{[2]}\subset (\PP^4)^*$ is the union of $10$ projective planes $\PP^2$ in $(\PP^4)^*$ and the surface $X^{[2]}_{(B)}$. Therefore, $\deg X^{[2]} = 235$.\\
The dual variety $\left( X^{[2]} \right)^* \subset \PP^4$ is the union of $10$ projective lines $\PP^1$ and $(X^{[2]}_{(B)})^*$. 
\end{proposition}

\begin{remark}
Let $X\subset \PP^4$ be a Bordiga surface and let $Y = X^{[2]}_{(B)}\subset X^{[2]}$ be the subvariety of degree $225$ of $X^{[2]}_{(B)}$. We would like to compute the degree of $Y^*$. Here, we want to discuss why the approach in Proposition~\ref{prop:Bordiga_3} for $k=3$ does not immediately adapt to this surface case.

The dual variety $Y^*$ is the closure of the union of all lines $\ell\subset \PP^4$ that are spanned by two points $p,q\in X$ such that $T_p X$ and $T_q X$ are contained in a common hyperplane $H$, where $[H]\in Y$. This is a ruling of $Y^*$ by a surface $S\subset {\textrm Gr}(2,5)$ of lines in $\PP^4$. The surface $S^\perp = \{\ell^\perp \colon [\ell]\in S\}\subset {\textrm{Gr}}(3,5)$ of all planes that are dual to lines in $S$ is, by biduality, the image of the map $\gamma\colon Y \dashrightarrow {\textrm{Gr}}(3,5)$, $[H]\mapsto T_{[H]} Y$.

The degree of $Y^*$ in terms of the Grassmannian is the number of lines $[\ell]\in S$ that intersect a fixed generic line $L\subset \PP^4$. Dually, this condition counts the cardinality of the set of all planes $[\ell^\perp]\in S^\perp$ that intersect the generic plane $L^\perp$ in at least a $1$-dimensional projective space (dual to the plane spanned by $L$ and such an $\ell$ intersecting $L$). The problem is now to compute the degree of this surface $S^\perp$ in the Plücker embedding. A central issue is that $Y$ is not known to be smooth (and we do not expect it to be). For this approach to succeed, we would need to have a good understanding of the tangent bundle of a desingularization of $Y$.
\end{remark}

\subsubsection{The irreducible components for k=3}
The curve $X^{[3]}$ is the closure of the set of quartics $Q_u$ such that $\widetilde{Q}_u$ has three nodes. The relevant partitions are $3 = 1\cdot 2 + 1\cdot 1$ (case (C)) and $3 = 3\cdot 1$ (cases (A) and (B)). We first discuss the reducible cases. The quartic cannot be the union of two conics since they intersect in $4$ points and are therefore not generic points of any irreducible component of $X^{[3]}$.

\paragraph{(A): cubic + line.} There are two possible configurations for a reducible plane quartic $Q_u$ through ten fixed points to have three nodes. Firstly, we have the one-dimensional family of curves $Q_u$ that are the union of a line through two of the ten points with a cubic through the remaining eight. These are all linear conditions on the coefficients of $Q_u$, hence each choice of the pair of points that fixes the line, produces a projective line $\PP^1\subset (\PP^4)^*$. There are $\binom{10}{2} = 45$ such irreducible components of $X^{[3]}$. 

The second possibility arises when the curves $Q_u$ are obtained by fixing the cubic through nine points, and taking its union with a line through the remaining point. By genericity of the base points, the cubic itself is not singular.
So, also in this case we get a projective line $\PP^1\subset (\PP^4)^*$ for each of the $10$ choices of the pencil of lines through a base point $p_i$.
Since a plane cubic intersects a line in three points, the cubic itself cannot be singular for a generic point $u$ on an irreducible component of $X^{[3]}$.

Therefore, these components contribute $45+10$ degrees to $X^{[3]}$. However, they are all projective lines, so their duals are projective planes which have codimension $2$. They are not part of the algebraic boundary of $\conv X$.

\paragraph{(B): rational quartics.} Assume $Q_u$ is a plane quartic with three nodes; the genus-degree formula implies that $Q_u$ is rational. The (Zariski closure of the) family of the corresponding $u$'s is then a Severi variety. The condition that all $Q_u$ must go through $p_1,\ldots,p_{10}$ makes it a curve in $\PP^4$, that we denote by $C_4$, to be consistent with the notation in \cite{Panda:CanonicalClass}. Its degree is classically known to be $620$ \cite{KonMan:GromovWittenClasses}. 
This number can be also computed using Kazarian's formulae.
Let $D = 4L$, $K_{\PP^2} = -3L$; then the degree of the curve of plane quartics with four nodes is $N_{A_1^4} = 675$. However, $55$ degrees come from case (A), so we have to subtract them and we get the correct number $620$. 

We will now compute the degree of the hypersurface dual to $C_4$. This curve was studied extensively by Pandharipande: it is an irreducible curve of arithmetic genus $5447$ and geometric genus $725$, which were computed (and denoted as $g_4$ and $\Tilde{g}_4$ respectively) in \cite[Section 3]{Panda:CanonicalClass}.
We reduce the computation of the degree of this dual variety to the Riemann-Hurwitz formula with the following well-known observation. 

\begin{proposition}
Let $C\subset \PP^n$ be an irreducible curve and let $\mathcal{T}$ be its tangent developable (that is the Zariski closure of the union of all tangent lines to $C$ at smooth points). Then, the degree of $\mathcal{T}$ and the degree of the dual variety $C^*$ agree.
\end{proposition}

\begin{proof}
The tangent developable $\mathcal{T}$ of $C$ is a surface in $\PP^n$. Its degree is obtained by intersecting it with a generic subspace $U\subset \PP^n$ of codimension $2$. On the other hand, $C^*$ is a hypersurface and its degree is computed by intersecting it with a generic line $L\subset (\PP^n)^*$. A generic subspace $U\subset \PP^n$ of codimension $2$ is dual to a generic line $L\subset (\PP^n)^*$. An intersection point of $L$ with $C^*$ corresponds to a hyperplane $H\subset \PP^n$ that is tangent to $C$ at a smooth point $x$. Since this hyperplane $H$ moves in the pencil of hyperplanes containing $L^* = U$, the tangent line to $C$ at $x$ intersects $U$. Conversely, an intersection point $x$ of $U$ with the tangent developable $\mathcal{T}$ of $C$ gives rise to a tangent hyperplane $H$ to $C$, namely the span of $U$ and the tangent line to $C$, on which $x$ lies. So this $H$ is an intersection point of $L$ with $C^*$. This one-to-one correspondence between intersection points of $U$ with $\mathcal{T}$ and intersection points of $L$ with $C^*$ shows the claim that $\deg \mathcal{T}  = \deg C^*$.
\end{proof}

\begin{remark}
The tangent developable $\mathcal{T}$ gives rise to a curve in $\mathbf{Gr}(2,n+1)$, and $C^*$ to a curve in $\mathbf{Gr}(n-1,n+1)$. Duality sets up a natural isomorphism $\mathbf{Gr}(2,n+1) \to \mathbf{Gr}(n-1,n+1)$ that restricts to an isomorphism $\mathcal{T}\to C^*$.
\end{remark}

Now we compute $\deg \mathcal{T}$ via the Riemann-Hurwitz formula: consider the projection of $C_4$ from a generic projective plane to $\PP^1$; each tangent line that intersects the chosen plane gives a ramification point of the projection. The degree of the projection is $\deg C_4 = 620$. Pandharipande's work provides an accurate description of the singularities of $C_4$. From that, we deduce that the ramification points of the map arise either from points of $C_4$ at which the tangent line intersects the given plane, or as projections of the cusps of $C_4$. The cusps of this curve correspond exactly to plane quartics $Q_u$ with two nodes and a cusp \cite[Lemma 3]{Panda:CanonicalClass}. There are $N_{A_1^2 A_2}=2304$ curves with that property that go through $p_1,\ldots,p_{10}$. We now put together all these information in the Riemann-Hurwitz formula:
\begin{equation}
    2\cdot 725 - 2 = 620 \cdot (-2) + (2304 + \deg \mathcal{T}),
\end{equation}
from which we obtain that $\deg C_4^* = \deg \mathcal{T} = 384$.

\paragraph{(C): a node at $p_i$.}
To $X^{[3]}$ belong also those points $u$ such that $Q_u$ has a node at $p_i$, for a fixed $i$, and another node in the plane. Every choice of $p_i$ gives rise to an irreducible component. Using Kazarian's formulae with
\begin{equation}\label{eq:divisors_X3_Bordiga_node}
    D = 4\pi^*(L) - 2E_i - \sum_{j\neq i} E_j, \qquad\qquad K_X = -3\pi^*(L) + \sum_{j=1}^{10} E_j,
\end{equation}
on $\mathrm{Bl}_{p_1,\ldots,p_{10}} \PP^2$, i.e., with parameters $\mathsf{d} = 3, \mathsf{k} = -1, \mathsf{s} = -1, \mathsf{x} = 13$, we find that the degree of each of these components is $N_{A_1} = 20$. In total, the contribution of case (C) to $\deg X^{[3]}$ is $20\cdot 10 = 200$. We discuss in the next section the dual varieties of these curves and their degree.

We describe in this section the singularities of the ten irreducible components of $X^{[3]}$ arising from case (C). From the analysis of the singularities we will deduce the degree of their dual varieties.
Let $Y\subset (\PP^4)^*$ be the Zariski closure of the points $u$ corresponding to plane quartics $Q_u$ with a node at $p_1$ and another node somewhere in $\PP^2$. We have the following description.
\begin{theorem}\label{thm:sing_Y}
The curve $Y$ is a plane curve of degree $20$ and geometric genus $9$ with $114$ nodes and $48$ cusps.
A point $u\in Y$ is a node if and only if $Q_u$ has a node at $p_1$ and two other nodes in $\PP^2$. A point $u\in Y$ is a cusp if and only if $Q_u$ has a node at $p_1$ and a cusp in $\PP^2$.
The plane curve dual to $Y$ has degree $8$ and $12$ cusps and no nodes.
\end{theorem}

\begin{proof}
It will be convenient to change the point of view and introduce an appropriate incidence variety. Let us first define the ambient space. Notice that the curve $\widetilde{Q}_u$ always contains the exceptional line $E_1$, since $Q_u$ has a node at $p_1$. Hence, the $u$'s that belong to $Y$ must satisfy $u^\perp \supset E_1$. This condition on the points cuts out a projective plane in $(\PP^4)^*$ that we denote by $\mathcal{H}_{E_1}$. Therefore, $Y$ is a plane curve in $\mathcal{H}_{E_1}\subset (\PP^4)^*$.

The curves $\widetilde{Q}_u$ for $u\in \mathcal{H}_{E_1}$ form the (base-point free) linear system associated to the divisor $D = 4L - 2E_1 - \sum_{i=2}^{10} E_i$ on $X$. The associated line bundle has $3$ global sections and the associated morphism to $\PP^2$ has degree $D^2 = 16 - 4 - 9 = 3$. We denote this $3:1$ cover of $\PP^2$ by $f$. Our curve $Y$ turns out to be the dual curve to the branch divisor of $f$, which is described in detail in \cite[Section 10]{miranda:triple}. 

The Tschirnhausen bundle associated to $f$ in \cite{miranda:triple} is the bundle $E = \mathcal{O}_{\PP^2}(-2)\oplus \mathcal{O}_{\PP^2}(-2)$ of rank $2$ on $\PP^2$, compare ibid.,~Table~10.5. Indeed, the associated surface $S$ is the blow up of $\PP^1 \times \PP^1$ in $9$ points. The pullback of a line in $\PP^2$ corresponds to a curve of bidegree $(2,3)$ in $\PP^1\times \PP^1$ through the nine basepoints. This surface $S$ is isomorphic to our Bordiga surface via the isomorphism of the blow up of $\PP^1\times \PP^1$ in one point and the blow up of $\PP^2$ in two points since $2(L-E_2) + 3(L-E_1) - (L-E_1-E_2) - \sum_{i=3}^{10} E_i = 4L - 2E_1 - \sum_{i=2}^{10}E_i$. This decomposition is meaningful because the $(-1)$-curve $L-E_1-E_2$ in the blow up of $\PP^2$ in two points corresponds to the exceptional divisor of the blow up of $\PP^1\times \PP^1$ in one point. Miranda's results give us the numbers in Theorem~\ref{thm:sing_Y} via the branch divisor $B$ of $f$. First, the degree of the branch divisor $B\subset \PP^2$ is $8$ by \cite[Proposition 4.7]{miranda:triple}.
Secondly, its singularities are determined by ibid.,~Lemma~4.8, which says that every singularity of the branch divisor $B$ of $f$ is a point of total ramification, and ibid.,~Corollary~5.8(ii), which shows that such a total ramificaction point is a double point of $B$ with one tangent and therefore generally an ordinary cusp. The number of such cusps is then counted in ibid.,~Lemma~10.1 as $3c_2(E)$ in terms of the second Chern class of the Tschirnhausen bundle $E$. Since this bundle is split in our case, we can compute the Chern classes of $E$ as the elementary symmetric polynomials in the Chern classes of the line bundle $\mathcal{O}_{\PP^2}(-2)$ so that $c_1(E) = -4L$ and $c_2(E) = 4 L^2$ in the Chow ring of $\PP^2$, see \cite[Corollary 5.4]{EisHar:3264}. So in summary, the curve $B$ has degree $8$, $12$ cusps, and genus $9$.

We show next that $Y$ is indeed the dual curve to the branch divisor $B$ of $f$. The pullback of a line $L$ in $\PP^2$ via $f$ is the curve $\widetilde{Q}_u$ for some appropriate $u\in \mathcal{H}_{E_1}$ and the restriction of $f$ to $Q_u$ is a triple cover of $\PP^1$. Since $Q_u$ has genus $2$ and the branch divisor has degree $8$, the curve $Q_u$ is smooth if the intersection of the line $L$ and the branch divisor $B$ is transversal by Riemann-Hurwitz. If the line $L$ is tangent to $B$ at a smooth point $p\in B$, then we choose a local coordinate $x$ on $L$ so that $p=0$. By Miranda's results, the fiber $f^{-1}(p)$ contains two points, one of them double. Locally (in an analytic sense) around the double point, we get a $2:1$ cover of an affine line that we can write as $z^2 = a_2x^2 + a_3x^3 + \ldots$. Here, the term on the right hand side starts with a quadratic term because $L$ is tangent to $B$ at $p$. If $a_2\neq 0$, this shows that the curve $Q_u$ has a node at the double point in $f^{-1}(p)$. If the line $L$ is even a flex line, then $a_3$ is also $0$ and the double point in $f^{-1}(p)$ is a cusp of $Q_u$. This shows both that $Y$ is the dual curve of $B$ and moreover it gives the characterization of the singularities of $Y$ in the theorem. Their numbers are determined by Plücker's formulae in terms of the degree and the singularities of $B$.
If $d$, $g$, $\delta$, and $\kappa$ denote the degree, geometric genus, number of nodes, and number of cusps for a plane curve $C$ and dually $d^*$, $g^*$, $\delta^*$, and $\kappa^*$ the invariants of $C^*$ then
\begin{align*}
    & d^* = d(d-1) - 2\delta - 3\kappa \\
    & d = d^*(d^*-1) - 2\delta^* - 3 \kappa^*\\
    & g = g^* =\frac12 (d^*-1)(d^*-2) - \delta^* - \kappa^*
\end{align*}
holds, see \cite{piene:pluecker}. So with $Y = B^*$ as well as $d=8$, $\delta =0$, and $\kappa = 12$, we find $d^* = 20$, $\delta^* = 114$, and $\kappa^* = 48$.
\end{proof}

The dual of the curve $Y$ as a curve in $\PP^4$ is a cone with a line of vertices that is dual to the plane containing $Y$. So we get the following result.
\begin{corollary}
    The variety $\left( X_{(C)}^{[3]} \right)^* \subset \PP^4$ is the union of $10$ hypersurfaces, arising as cones over plane curves, each of degree $8$. 
\end{corollary}

\begin{remark}
We can also count the singularities of $Y$ by counting plane quartics $Q_u$ that are more singular. More precisely, the number of nodes of $Y$ is the number of plane quartics with a node at $p_1$ and two other nodes on $\PP^2$; the number of cusps of $Y$ is the number of plane quartics with a node at $p_1$ and a cusp somewhere in $\PP^2$. They are Kazarian's numbers $N_{A_1^2}=114$ and $N_{A_2}=48$ respectively, for the divisors in \eqref{eq:divisors_X3_Bordiga_node}.
\end{remark}

We are now done with the case $k=3$. Indeed, assume that the quartic $Q_u$ is irreducible. Then it can have three nodes somewhere in $\PP^2$, that make three nodes also on $\widetilde{Q}_u$ (case (B)). Otherwise, we can impose that $Q_u$ has a node at one of the fixed points $p_i$, and another node in the plane, so that the pullback has three singularities (case (C)). There are no other options for irreducible quartics. If $Q_u$ is reducible, we have different cases. Assume that it is the union of a cubic and a line; then this makes three singularities (case (A)). If $Q_u$ is the union of two conics, or one conic and two lines, or four lines, then there are too many singularities. So this case is not of interest for $k=3$. All the possible cases have been discussed above. We summarize our findings for $X^{[3]}$ and $(X^{[3]})^*$ in the following proposition.

\begin{proposition}\label{prop:Bordiga_3}
Let $X\subset \PP^4$ be a generic Bordiga surface. Then $X^{[3]}\subset (\PP^4)^*$ is the union of $66$ irreducible components. Among them, $55$ are projective lines $\PP^1$, one is a Severi variety of degree $620$ and the remaining $10$ are cones over plane curves of degree $20$. Therefore, we have $\deg X^{[3]} = 675$.\\
The dual variety $\left( X^{[3]}\right)^*\subset \PP^4$ is the union of $55$ projective planes $\PP^2$, an irreducible variety of degree $384$ and $10$ irreducible varieties of degree $8$. Only the last two types of components are hypersurfaces, with total degree $464$.
\end{proposition}

\subsubsection{The irreducible components for k=4}

We are left to compute the cardinality of $X^{[4]}$, whose points correspond to plane quartics $Q_u$ such that $\widetilde{Q}_u$ has $4$ nodes. The partitions are $4 = 2\cdot 2$ (case (C)), $4 = 1\cdot 2 + 2\cdot 1$ (case (B)), and $4 = 4\cdot 1$ (case (A)).

\paragraph{(A): four nodes in $\PP^2$.} Assume that $Q_u$ is a plane quartic with four nodes. Then it must be reducible and two configurations are possible: the union of two conics through $5$ points each, or the union of a singular cubic through eight points and a line through the other two. In the first situation we get $\frac{1}{2}\cdot\binom{10}{5} = 126$ quartics. In the second one, whenever we choose $8$ points, there are $12$ singular plane cubics through them. The number $12$ can be understood as the degree of the discriminant of cubics, or it can be computed with Kazarian's formulae for $D=3L$, $K_{\PP^2} = -3$ on $\PP^2$, i.e., $\mathsf{d} = 9, \mathsf{k} = -9, \mathsf{s} = 9, \mathsf{x} = 3$.
Hence, we have $\binom{10}{8} \cdot 12 = 540$ such quartics $Q_u$.
In total, this case contributes $666$ points of $X^{[4]}$, in agreement with Kazarian's number $N_{A_1^4}$ for the system of plane quartics with four nodes.

\paragraph{(B): a node at $p_i$.} Consider plane quartics with a node at one of the $p_i$'s and other two nodes somewhere in $\PP^2$. Once we fix $p_i$, we can count this number as $N_{A_1^2}$  for
\begin{equation}
    D = 4\pi^*(L) - 2E_i - \sum_{j\neq i} E_j, \qquad\qquad K_X = -3\pi^*(L) + \sum_{j=1}^{10} E_j,
\end{equation}
i.e., with parameters $\mathsf{d} = 3, \mathsf{k} = -1, \mathsf{s} = -1, \mathsf{x} = 13$.
We get $N_{A_1^2} = 114$ using Kazarian's formulae, thus this case contributes with $10\cdot 114 = 1140$ points to $X^{[4]}$.

\paragraph{(C): two nodes at $p_i$, $p_j$.} Assume the plane quartic $Q_u$ has a node at $p_i$ and a node at $p_j$. There is only one choice for such a curve. Therefore, this case contributes with as many points as the choices of $i,j$, i.e., $\binom{10}{2}=45$.

Note that in case (B) we also count case (C). Indeed, $114$ are the quartics on the Bordiga surface coming from a curve with a node at $p_i$, and two more nodes outside $E_i$. Curves $\widetilde{Q}_u$ of case (C), where $Q_u$ has nodes at $p_i, p_j$, belong to case (B) twice: one time for $p_i$ and one for $p_j$. In total, we have $\# X^{[4]} = 666+(1140-90)+45 = 1761$, as predicted in Table \ref{table:degreesXk}. 

There are no other possible ways to obtain four singularities on $\widetilde{Q}_u$. If $Q_u$ is irreducible, then it must have one or two nodes at the $p_i$'s and two or one other node in the plane, respectively (cases (B) and (C)); four nodes in $\PP^2$ would make it reducible. If $Q_u$ is reducible, then there are two cases (both in case (A)): it is either the union of a cubic with a line, in which case the cubic must have a node, or it is the union of two conics. The configuration with a conic and two lines or the one with four lines produce too many singularities, so we do not consider them. We summarize our findings for $X^{[4]}$ and $(X^{[4]})^*$ in the following proposition.

\begin{proposition}\label{prop:Bordiga_4}
Let $X\subset\PP^4$ be a generic Bordiga surface. Then $X^{[4]}\subset (\PP^4)^*$ is the union of $1761$ points. Therefore, $\left( X^{[4]} \right)^*\subset \PP^4$ is the union of $1761$ hyperplanes.
\end{proposition}

\subsection[The real picture]{The $\R$eal picture}
Again, we want to dive in the realm of convex and real algebraic geometry in this section. Here, we only consider the real forms of Bordiga surfaces arising as blow ups of the real projective plane. 
To such a Bordiga surface we can associate a $3\times 4\times 5$ tensor. A generic such tensor gives a smooth Bordiga surface. By considering the matrix space with respect to the third factor, we get a $3\times 4$ matrix with linear entries in $5$ variables. The $3\times 3$ minors of this matrix are the equations defining the Bordiga surface in $\PP^4$. By considering the matrix space with respect to the first factor, we get a $4\times 5$ matrix with linear entries in $3$ variables. The zero locus of the maximal minors of this matrix define $10$ points in $\PP^2$: the base points of the blow up that gives $X$.
The Bordiga surface whose associated tensor has real entries are certainly real algebraic surfaces. As a consequence, the ten base points are either real or they come in complex conjugate pairs. However, if one of the points, say $p_1$, is real, then $X$ contains the real line $E_1$. Therefore, the real locus of $X$ is not compact on any affine chart of $\PP^4$. So the relevant case for us is when $p_1,\ldots,p_{10}$ are five pairs of complex conjugate points in $\PP^2$.

The central question is again: Which of the $\left( X^{[k]} \right)^*$ can contribute to $\partial_a \conv X$?
From the blow up construction of $X$ we can deduce that $u^\perp$ is a supporting hyperplane if and only if the polynomial $q_u$ defining the quartic $Q_u\subset \PP^2$ has constant sign on $\PP^2$. We are going to examine the points of the $X^{[k]}$ to understand if such a property is realizable in the various cases. We will focus only on those components of the $X^{[k]}$ whose duals are hypersurfaces and so have a chance to be part of the algebraic boundary. This analysis will affect the choices of the ten base points $p_1,\ldots, p_{10}$. The results will be summarized in Theorem \ref{thm:bordiga}.

We start with $X^{[2]}$. The only candidate here is case (B), i.e., a plane quartic through $p_1,\ldots,p_{10}$ with two nodes in $\PP^2$. In order to show that this configuration can produce supporting hyperplanes, we will construct first the hyperplane section, and we will choose the Bordiga surface in a second moment. Define the quartic polynomial $q_u$ in the following way. Pick two points $x,y$ in the plane and let $c_1, c_2, c_3$ be homogeneous quadratic polynomials in three variables defining three conics that intersect in $x,y$. Define $q_u = c_1^2+c_2^2+c_3^2$. Since it is a sum of squares, $q_u$ is non-negative on the whole $\PP^2$, and zero only at the two singular points $x,y$. Hence $u^\perp$ is a separating hyperplane for $X$, where $X$ is the blow up of $\PP^2$ in ten points that are five pairs of complex conjugate zeros of $q_u$.

Regarding $X^{[3]}$, we have to analyse both case (B) and (C). Case (B) is given by rational quartics through $p_1,\ldots,p_{10}$, which have three nodes $x,y,z$ in $\PP^2$. We can repeat the same procedure as for $X^{[2]}_{(B)}$ to construct a separating hyperplane; the only difference is that now the three conics will have to intersect in the three nodes $x,y,z$.
The points in case (C) represent quartics $Q_u$ with a node at one of the ten points, say $p_1$, and another node $x\in\PP^2$. Since we want the hyperplanes coming from $X^{[k]}$ to support $X$ in $k$ real points, then in this case $p_1$ is forced to have real coordinates. This implies that there is also another base point, say $p_2$ which is real. However, $p_2$ is a smooth point for the quartic $Q_u$. Therefore, in a neighborhood of $p_2$ the polynomial $q_u$ changes sign. This implies that these $u^\perp$ cannot be a supporting hyperplane for $X$, independently of the choice of the $p_i$.

Finally, $X^{[4]}$ is subdivided in three cases. In case (A), the quartics $Q_u$ are the union of either two conics or a line and a cubic. In both cases, there is no way to make just the four singular points real, hence these reducible quartics change sign in $\PP^2$. So the associated hyperplanes cannot support $X$, independently of the choice of the $p_i$. 
Case (B) is similar to the variety $X^{[3]}_{(C)}$, since we have a node at, say, $p_1$ and therefore $p_2$ must be real as well. However, it is a smooth point of $Q_u$, and therefore there is a sign change that prevent $u^\perp$ from being a supporting hyperplane, independently of the choice of the $p_i$. 
The points of case (C) correspond to quartics with two nodes at, say, $p_1,p_2$. Therefore both of them must be real. In this setting, we can consider three conics $c_1,c_2,c_3$ that intersect exactly at $p_1,p_2$. Then, define $q_u = c_1^2+c_2^2+c_3^2$. This provides a non-negative quartic on $\PP^2$ which is zero precisely at $p_1,p_2$. Hence $u^\perp$ is a separating hyperplane for $X$, where $X$ is the blow up of $\PP^2$ in ten points that are $p_1,p_2$ and four pairs of complex conjugate zeros of $q_u$. In particular, this shows that just one hyperplane of $X^{[4]}_{(C)}$ can be a supporting hyperplane of $X$.

The following statement summarizes the whole discussion on the Bordiga case.
\begin{theorem}\label{thm:bordiga}
Let $X\subset\PP^4$ be a smooth Bordiga surface with compact real locus in an appropriate affine chart, arising as the blow up of $\PP^2$ at ten base points $p_1,\ldots,p_{10}$ which are five generic pairs of complex conjugates. Then
    \begin{equation}
        \partial_a (\conv X) \subset \left( X^{[2]}_{(B)} \right)^* \cup \left( X^{[3]}_{(B)} \right)^*, 
        \qquad 
        \deg \partial_a (\conv X) \leq \deg \left( X^{[2]}_{(B)} \right)^* + 384.
    \end{equation}
\end{theorem}

It is natural to ask whether equality in the formulae in Theorem \ref{thm:bordiga} can be achieved. The answer is positive, as shown in the following example. 
\begin{example}
Consider the two homogeneous quartic polynomials
\begin{align}
    \begin{split}
        f_2 &= (2 x_0^2 + x_0 x_1 + 2 x_0 x_2 - x_1^2)^2 \!+\! (x_0^2 - x_1^2 - x_2^2)^2 \!+\! (x_0^2 - 2 x_0 x_2 - x_1^2 - x_1 x_2 - 3 x_2^2)^2,
    \end{split}\\
    \begin{split}
    f_3 &= (x_0 x_1 + x_0 x_2 - x_1^2 - x_2^2)^2 \!+\! (x_0 x_1 + 4 x_0 x_2 - x_1^2 - 4 x_1 x_2 - 4 x_2^2)^2\\ 
    &\hspace{7cm} + (2 x_0 x_1 - 2 x_0 x_2 - 2 x_1^2 - x_1 x_2 + 2 x_2^2)^2.
    \end{split}
\end{align}
They are the sum of the squares of three quadratic polynomials, defining plane conics. The three conics of $f_i$ intersect in $i$ real points, as shown in Figure \ref{fig:f_i}. Therefore, the real locus of the quartic $\{f_i=0\}$ consists of $i$ nodes, the green dots in Figure \ref{fig:f_i}.
\begin{figure}[bh]
    \centering
    \begin{tikzpicture}
\begin{axis}[
width=2.3in,
height=2.3in,
scale only axis,
xmin=-1.5,
xmax=1.5,
ymin=-1.5,
ymax=1.5,
ticks = none, 
ticks = none,
axis background/.style={fill=white},
axis x line = middle, 
axis y line = middle, 
axis line style={draw=none},
samples=400
]

\addplot[domain=-2:2, line width=0.3mm, color=black] {1/2*x^2-1/2*x-1};
\addplot[domain=-1.03785211:1.40149, line width=0.3mm, color=black] {1/6*(-2-x+sqrt(16+4*x-11*x^2))};
\addplot[domain=-1.03785211:1.40149, line width=0.3mm, color=black] {1/6*(-2-x-sqrt(16+4*x-11*x^2))};
\draw[line width=0.3mm, color=black] (0,0) circle (1);

\addplot[only marks,mark=*,mark size=3pt,mygreen,
]  coordinates {
    (0,-1) (-1,0)
};
\end{axis}
\end{tikzpicture}
    \qquad
    \begin{tikzpicture}
\begin{axis}[
width=2.3in,
height=2.3in,
scale only axis,
xmin=-1.5,
xmax=1.5,
ymin=-1.5,
ymax=1.5,
ticks = none, 
ticks = none,
axis background/.style={fill=white},
axis x line = middle, 
axis y line = middle, 
axis line style={draw=none},
samples=400
]

\addplot[domain=-8:0.99, line width=0.3mm, color=black] {1/2*(1-x-sqrt(1-x))};
\addplot[domain=-8:0.99, line width=0.3mm, color=black] {1/2*(1-x+sqrt(1-x))};
\addplot[domain=-1.5:1.5, line width=0.3mm, color=black] {1/4*(2+x-sqrt(4-12*x+17*x^2))};
\addplot[domain=-1.5:1.5, line width=0.3mm, color=black] {1/4*(2+x+sqrt(4-12*x+17*x^2))};
\draw[line width=0.3mm, color=black] (0.5,0.5) circle (0.7071);

\addplot[only marks,mark=*,mark size=3pt,mygreen,
]  coordinates {
    (0,1) (1,0) (0,0)
};
\end{axis}
\end{tikzpicture}
    \caption{The construction of the polynomials $f_i$. Left: the three conics that are summands of $f_2$, intersecting at $(-1,0), (0,-1)$. Right: the three conics that are summands of $f_3$, intersecting at $(0,0), (1,0), (0,1)$. }
    \label{fig:f_i}
\end{figure}
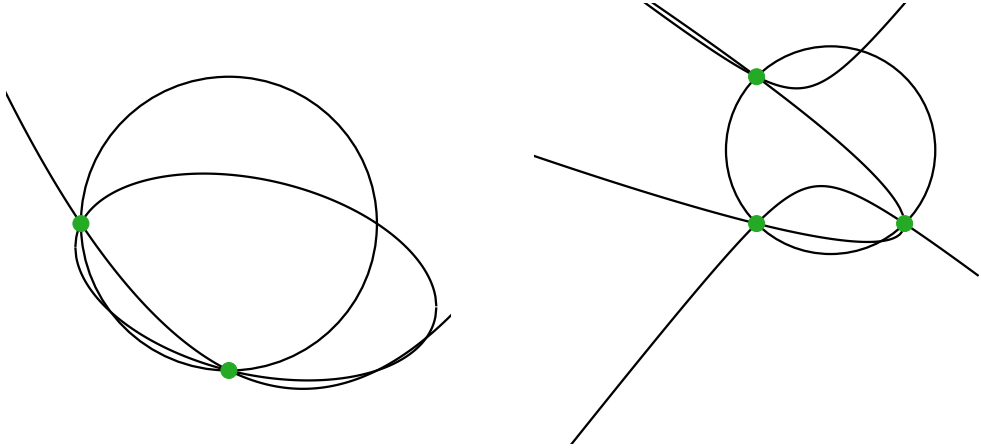
One can check that the set $\{f_2=f_3=0\}$ consists of $8$ pairs of complex conjugate points. Let $p_1,\ldots,p_{10}$ be five such pairs, and let $X = \textrm{Bl}_{p_1, \ldots ,p_{10}} \PP^2 \subset\PP^4$ be the associated Bordiga surface. Then there exist $u_2,u_3 \in (\PP^4)^*$ such that $Q_{u_i} = \{f_i = 0\}$. By construction, this implies that $u_2 \in X^{[2]}_{(B)}$ and $u_3 \in X^{[3]}_{(B)}$. Moreover, since the quartics $f_i$ are sums of squares, the corresponding hyperplanes separate $X$ and therefore belong to the algebraic boundary. Then, the algebraic boundary consists of two irreducible components:
\begin{equation}
    \partial_a(\conv X) = \left( X^{[2]}_{(B)} \right)^* \cup \left( X^{[3]}_{(B)} \right)^*.
\end{equation} 
\end{example}

\bibliographystyle{alpha}
\bibliography{biblio}

\end{document}